\documentclass[12pt,reqno]{amsart}
\usepackage[numbers, square]{natbib}
\usepackage{mathptmx}
\usepackage{amsmath}
\usepackage{amscd}
\usepackage{amssymb}
\usepackage{amsthm}
\usepackage{enumerate}
\usepackage{xspace}
\usepackage[all,tips]{xy}
\usepackage[dvips]{graphicx}
\usepackage{verbatim}
\usepackage{syntonly}
\usepackage{hyperref}
\usepackage{amsmath, amsthm, graphics, amssymb,fullpage,color, epsfig,url}
\usepackage{indentfirst}
\usepackage{color}
\usepackage{mathtools}

\theoremstyle{plain}
\newtheorem{thm}{Theorem}
\newtheorem{lem}{Lemma}
\newtheorem{prop}{Proposition}
\newtheorem{defn}{Definition}

\newtheorem{clm}{Claim}

\theoremstyle{definition}
\newtheorem{rem}{Remark}

\newtheorem*{thm0}{Theorem}
\newtheorem*{prop0}{Proposition}

\newcommand{\disp}{\displaystyle}

\DeclareMathOperator{\di}{div}

\DeclareMathOperator{\tr}{tr}

\newcommand{\eps}{\varepsilon}
\newcommand{\vp}{\varphi}

\newcommand{\al}{\alpha}
\newcommand{\be}{\beta}
\newcommand{\ga}{\gamma}
\newcommand{\de}{\delta}

\newcommand{\Ga}{\Gamma}
\newcommand{\te}{\theta}
\newcommand{\la}{\lambda}
\newcommand{\La}{\Lambda}
\newcommand{\om}{\omega}
\newcommand{\Om}{\Omega}
\newcommand{\si}{\sigma}

\newcommand{\iny}{\infty}
\newcommand{\del}{ \partial}
\newcommand{\su}{\subset}
\newcommand{\LP}{\Delta}
\newcommand{\gr}{\nabla}

\newcommand{\norm}[1]{\left\| #1\right\|}

\newcommand{\innp}[1]{\left< #1 \right>}
\newcommand{\abs}[1]{\left\vert#1\right\vert}
\newcommand{\set}[1]{\left\{#1\right\}}
\newcommand{\brac}[1]{\left[#1\right]}
\newcommand{\pr}[1]{\left( #1 \right) }
\newcommand{\pb}[1]{\left( #1 \right] }

\DeclarePairedDelimiter{\ceil}{\lceil}{\rceil}

\newcommand{\N}{\ensuremath{\mathbb{N}}}

\newcommand{\R}{\ensuremath{\mathbb{R}}}

\numberwithin{equation}{section}
\numberwithin{defn}{section}
\numberwithin{lem}{section}
\numberwithin{cor}{section}

\newcommand{\Keywords}[1]{\par\noindent 
{\small{\bf Keywords\/}: #1}}
\newcommand{\MSC}[1]{\par\noindent 
{\small{\bf Mathematics Subject Classification\/}: #1}}

\date{}

\title{On Landis' conjecture in the plane for some equations \\ with sign-changing potentials}

\author[Davey]{Blair Davey}
\address{Department of Mathematics, City College of New York CUNY, New York, NY 10031, USA}
\email{bdavey@ccny.cuny.edu}
\thanks{Davey is supported in part by the Simons Foundation Grant 430198.}

\date{}

\begin{document}

\maketitle

\begin{abstract}
In this article, we investigate the quantitative unique continuation properties of real-valued solutions to elliptic equations in the plane.
Under a general set of assumptions on the operator, we establish quantitative forms of Landis' conjecture.
Of note, we prove a version of Landis' conjecture for solutions to $-\Delta u + V u = 0$, where $V$ is a bounded function whose negative part exhibits polynomial decay at infinity.
The main mechanism behind the proofs is an order of vanishing estimate in combination with an iteration scheme.
To prove the order of vanishing result, we present a new idea for constructing positive multipliers and use it reduce the equation to a Beltrami system.
The resulting first-order equation is analyzed using the similarity principle and the Hadamard three-quasi-circle theorem.
\\

 \\

\MSC{35B60, 35J10} 

\end{abstract}

\section{Introduction}

In the late 1960s, E.~M.~Landis conjectured that if $u$ is a bounded solution to $-\LP u + V u = 0$ in $\R^n$, where $V$ is a bounded function and $\abs{u(x)} \lesssim \exp\pr{- c \abs{x}^{1+}}$, then $u \equiv 0$.
This conjecture was later disproved by Meshkov who in \cite{Mes92} constructed non-trivial functions $u$ and $V$ that solve $-\LP u + V u = 0$ in $\R^2$, where $V$ is bounded and $\abs{u(x)} \lesssim \exp\pr{- c \abs{x}^{4/3}}$. 
Meshkov also proved a {\em qualitative unique continuation} result: if $-\LP u + V u = 0$ in $\R^n$, where $V$ is a bounded function and $\abs{u\pr{x}} \lesssim \exp\pr{- c \abs{x}^{4/3+}}$, then necessarily $u \equiv 0$.
In their work on Anderson localization \cite{BK05}, Bourgain and Kenig established a {\em quantitative} version of Meshkov's result. 
They showed that if $u$ and $V$ are bounded, and $u$ is a normalized solution for which $\abs{u(0)} \ge 1$, then a three-ball inequality derived from a Carleman estimate shows that for sufficiently large values of $R$,
\begin{equation}
 \inf_{|x_0| = R}\norm{u}_{L^\iny\pr{B(x_0, 1)}} \ge \exp{(-CR^{\be}\log R)},
\label{est}
\end{equation} 
where $\be = \frac 4 3$.
Since $ \frac 4 3 > 1$, the constructions of Meshkov, in combination with the qualitative and quantitative unique continuation theorems just described, indicate that Landis' conjecture cannot be true for complex-valued solutions in $\R^2$.
However, Landis' conjecture still remains open in the real-valued and higher-dimensional settings.

Here we prove a collection of quantitative unique continuation results for real-valued solutions to equations in the plane of the form
\begin{align}
\mathcal{L} u := - \di \pr{A \gr u} + W \cdot \gr u +  V u = 0,
\label{ePDE}
\end{align}
where the coefficient matrix $A = \pr{a_{ij}}_{ij=1,2}$ is assumed to be bounded and elliptic.
That is, there exists $\la > 0$ so that for all $z, \xi, \zeta \in \R^2$,
\begin{align}
&\la \abs{\xi}^2 \le a_{ij}\pr{z} \xi_i \xi_j 
\label{ellip} \\
& a_{ij}\pr{z} \xi_i \zeta_j \le \la^{-1} \abs{\xi} \abs{\zeta}. 
\label{Abd}
\end{align}
Further, $A$ is Lipschitz continuous with a decaying derivative.
This means that there exist constants $\mu_0, \eps_0 > 0$ so that
\begin{align}
& \abs{\gr a_{ij}\pr{z}} \le \mu_0 \innp{z}^{-\pr{1 + \eps_0}},
\label{Lips}
\end{align}
where we recall that $\innp{z} = \sqrt{1 + \abs{z}^2}$.
Moreover, we assume that $V, W \in L^\iny_{loc}\pr{\R^2}$.

Before stating the first main theorem, we introduce an important definition.

\begin{defn}
We say that the operator $\mathcal{L}$ is {\em non-negative} in $\Om \su \R^2$ if the cone of positive solutions to $\mathcal{L} u = 0$ in $\Om$ is non-empty.
\end{defn}

An important implication of this definition is that if $\mathcal{L}$ is non-negative, then there exists an eigenvalue $\la_0 \le 0$ and a positive, continuous function $\phi$ so that
$$\mathcal{L}\phi + \la_0 \phi = 0 \quad \text{ in } \Om.$$
In particular, $\phi$ is a positive supersolution.

We now state the main theorem.

\begin{thm}
\label{LandisThm0}
Let the coefficient matrix $A$ satisfy \eqref{ellip}, \eqref{Abd}, and \eqref{Lips}. 
Assume that $\norm{V}_{L^\iny\pr{\R^2}} \le \mu_1^2$, $\norm{W}_{L^\iny\pr{\R^2}} \le \mu_2$, and that there exists an exterior domain $\Om \su \R^2$ such that $\mathcal{L}$ is non-negative in $\Om$.
Let $u: \R^2 \to \R$ be a solution to \eqref{ePDE} for which 
\begin{align}
&\abs{u\pr{z}} \le C_0 \abs{z}^{c_0}
\label{uBd} \\
& \abs{u\pr{0}} \ge 1.
\label{normed}
\end{align}
Then for any $\epsilon > 0$ and any $R \ge R_0\pr{\la, \mu_0, \mu_1, \mu_2, \eps_0, \Om, C_0, c_0, \eps}$, 
\begin{equation}
\inf_{\abs{z_0} = R} \norm{u}_{L^\iny\pr{B_1\pr{z_0}}} \ge \exp\pr{- R^{1+\eps}}.
\label{globalEst}
\end{equation}
If $\Om = \R^2$, we can take $\eps_0 = -1$ in \eqref{Lips} and assume that $\abs{u\pr{z}} \le \exp\pr{C_0 \abs{z}}$ instead of \eqref{uBd}, then \eqref{est} holds for all $R \ge R_0\pr{\la}$ with $\be = 1$ and $C = C\pr{\la, \mu_0, \mu_1, \mu_2,C_0}$.
\end{thm}

Under a specific set of conditions on the lower order terms, we prove another version of this theorem.

\begin{thm}
\label{LandisThm}
Let the coefficient matrix $A$ satisfy \eqref{ellip}, \eqref{Abd}, and \eqref{Lips}. 
Assume that $V = V_+ - V_-$, $V_\pm : \R^2 \to \R_{\ge 0}$, $W : \R^2 \to \R^2$, and there exist constants $\mu_1, \mu_2, \eps_1, \eps_2 > 0$ so that 
\begin{align}
& \norm{V_+}_{L^\iny\pr{\R^2} } \le 1
\label{V+Cond} \\
&V_-\pr{z} \le \mu_1^2 \innp{z}^{-2\pr{1 + \eps_1}} 
\label{V-Cond} \\
& \abs{W\pr{z}} \le \mu_2 \innp{z}^{-\pr{1 + \eps_2}}.
\label{WCond}
\end{align}
Let $u: \R^2 \to \R$ be a solution to \eqref{ePDE} for which \eqref{uBd} and \eqref{normed} hold.
Then for any $\epsilon > 0$ and for any $R \ge R_0\pr{\la, \mu_0, \mu_1, \mu_2, \eps_0, \eps_1, \eps_2, C_0, c_0, \eps}$, \eqref{globalEst} holds.
\end{thm}

The main mechanism behind the proofs of these theorems is an order of vanishing result.
Once we prove such estimates, we use a scaling argument in combination with an iteration scheme similar to those in \cite{Dav14}, \cite{LW14} and \cite{DKW18} to prove each unique continuation at infinity theorem.

As is standard, we use the notation $B_r\pr{z}$ to denote a ball of radius $r$ centered at $z$, and occasionally write $B_r$ when the center of the ball is understood.
Recall from \cite{DKW17} and \cite{DW17} that $Q_s$ denotes a {\em quasi-ball} associated to a second-order elliptic operator in divergence form, $L$.
More precisely, $Q_s$ is a set in $\R^2$ whose boundary curve is an $s$-level set of the fundamental solution of $L$.
To accommodate for scaling considerations, $Q_s$ is defined so that if $L = \LP$, then $Q_s = B_s$.
Further details can be found in Section \ref{quasi}.
In the following order of vanishing theorem, the functions $\si$ and $\rho$ provide lower and upper  bounds on the radii of the quasi-balls.
More details on these functions may also be found in Section \ref{quasi}.

\begin{thm}
\label{OofV}
Let the coefficient matrix $A$ satisfy \eqref{ellip} and \eqref{Abd}.
Let $F$ be a function such that $K^{-1} \lesssim F\pr{K} < 1$ for all $K \ge 1$.
For some $K \ge 1$, define 
\begin{align}
& d = \si\pr{1 - F\pr{K}}
\label{dDef} \\
& b = \rho\pr{1 + F\pr{K}}
\label{bDef} \\
& m = b + F\pr{K}.
\label{mDef}
\end{align}
Assume that $\norm{\gr a_{ij}}_{L^\iny\pr{B_m}} \le K$, $\norm{V}_{L^\iny\pr{B_m}} \le K^2$, $\norm{W}_{L^\iny\pr{B_m}} \le K$, and that there exists a positive function $\phi$ that solves \eqref{ePDE} in $B_m$.
Let $u$ be a real-valued solution to \eqref{ePDE} in $B_m$ that satisfies
\begin{align}
& \norm{u}_{L^\iny\pr{B_{{m}}}} \le \exp\pr{C_1 K}
\label{localBd} \\
& \norm{u}_{L^\iny\pr{B_d}} \ge \exp\pr{- c_1 K^p},
\label{localNorm}
\end{align} 
for some $p \ge 0$.
Then for any $r$ sufficiently small, 
\begin{equation}
\norm{u}_{L^\iny\pr{B_r}} \ge r^{C K^q /F\pr{K}},
\label{localEst}
\end{equation}
where $q = \max\set{1, p}$ and $C = C\pr{\la, C_1, c_1}$.
\end{thm}

In this theorem, we require a positive solution.
We present two ways to show that such a function $\phi$ exists.
If $\mathcal{L}$ is non-negative, the corresponding supersolution can be used to show that a positive multiplier exists.
These details may be found in Lemma \ref{posMul0}.
Alternatively, under suitable conditions on the norms of $\gr a_{ij}$, $V_-$ and $W$, Lemma \ref{posMul} shows directly that such a function $\phi$ exists.

\begin{rem}
The condition given in \eqref{Lips} was assumed previously in \cite{Tu10} and \cite{LW14}.
Not only will we need this assumption to start the iteration process that proves Theorems \ref{LandisThm0} and \ref{LandisThm}, but it will be crucial to controlling the size of quasi-balls, and therefore ensuring that the iteration argument works.
Further details will be discussed in the proof below.
\end{rem}

\begin{rem}
\label{symm}
Note that we can write
\begin{align*}
\mathcal{L} = - \di\pr{\hat A \gr } + \hat W \cdot \gr + V,
\end{align*}
where $\hat A$ denotes the symmetrization of $A$, $\hat W = W +  \pr{\del_y \,\check a_{12}, -\del_x \,\check a_{12}}$, and $\check a_{12} = \frac{a_{12} - a_{21}} 2$.
It is clear that $\hat A$ satisfies all of the same conditions as $A$, namely \eqref{ellip}, \eqref{Abd}, and \eqref{Lips}.
Since condition \eqref{Lips} (even with $\eps_0 = -1$) implies that $\norm{\hat W}_{L^\iny\pr{\R^2}} \le \mu_2 + \mu_0$, then there is no loss of generality in assuming that $A$ is symmetric in Theorem \ref{LandisThm0}.
Similarly, condition \eqref{Lips} implies that $\abs{\hat W\pr{z}} \le \pr{\mu_2 + \mu_0} \innp{z}^{-\pr{1 + \hat \eps_2}}$, where $\hat \eps_2 = \min\set{\eps_0, \eps_2}$, so there is also no loss in assuming that $A$ is symmetric in Theorem \ref{LandisThm}.
Finally, since we assume that $\norm{\gr a_{ij}}_{L^\iny\pr{B_m}} \le K$ and $\norm{W}_{L^\iny\pr{B_m}} \le K$ in Theorem \ref{OofV}, then $\norm{\gr \hat a_{ij}}_{L^\iny\pr{B_m}} \le K$ and $\norm{\hat W}_{L^\iny\pr{B_m}} \le C K$ as well.
Therefore, we assume from now on that $A$ is symmetric since it will simplify many of the proofs.
For the lemmas that require symmetry, we will indicate this additional hypothesis.
\end{rem}

In recent years, considerable progress has been made towards resolving Landis' conjecture in the real-valued planar setting.
In their breakthrough article \cite{KSW15}, Kenig, Silvestre, and Wang introduced a new method based on tools from complex analysis to reduce the value of $\be$ in \eqref{est} from $4/3$ down to $1$. 
Using the scaling argument first introduced in \cite{BK05}, the Landis-type theorems in \cite{KSW15} are consequences of order of vanishing estimates for solutions to local versions of the equation.
We now recall the main steps involved in proving these local theorems.
Under the assumption that $V$ is bounded and a.e. non-negative, a positive multiplier is used to transform the PDE for $u$ into a divergence equation.
The stream function associated to this divergence-free equation is then used to produce a first-order complex-valued equation known as a Beltrami system.
The similarity principle for such equations, in combination with the Hadamard three-circle theorem, gives rise to a three-ball inequality that is much sharper than those produced previously using Carleman estimate techniques.

The methods from \cite{KSW15} have been used and generalized in subsequent years to prove many Landis-type theorems.
In \cite{DKW17}, we proved variable-coefficient versions of the theorems in \cite{KSW15} through the use of quasi-conformal transformations.
The results in \cite{DW17} apply to very general elliptic equations and rely on the theory of boundary value problems to produce positive multipliers.
In our most recent article, \cite{DKW18}, we considered potential functions that are not necessarily non-negative, but are allowed to have some rapidly decaying negative part.
To treat this setting, we studied the quantitative behavior of solutions to {\em vector-valued} Beltrami systems and established appropriate generalizations of the similarity principles that had been used previously.
For a more general survey of Landis' conjecture and other related unique continuation results, we refer the reader to the introduction of \cite{KSW15}.

In \cite{ABG19} and \cite{Ros18}, the authors proved qualitative unique continuation at infinity estimates for solutions to elliptic equations under the assumption that the principle eigenvalue is non-negative.
Theorem \ref{LandisThm0} was motivated by these works and may be interpreted as a quantitative version of their estimates.

Theorem \ref{LandisThm} improves upon our most recent results in a few ways.
First, instead of imposing the condition that $V_-$ must decay rapidly (exponentially) at infinity, we can now handle the case where $V_-$ exhibits slow (polynomial) decay at infinity.
Second, we allow the leading operator to be variable.
In particular, we assume that the coefficients are Lipschitz continuous with derivatives that decay slowly.
Finally, we allow the first-order term to be non-zero, but we require that it also decays slowly at infinity.
An important example within this framework is the equation
$$- \LP u + V u = 0,$$
where $V_+ \in L^\iny$ and $V_-\pr{z} \lesssim \innp{z}^{-N}$ for any $N > 2$.

There are two main challenges involved in proving a quantitative version of Landis' conjecture in the current setting.
First, because $V_-$ is no longer assumed to be non-trivial, a new idea is required to establish the existence of a positive multiplier associated to the equation.
In \cite{DKW18}, we defined a function $V_\de = V + \de^2$, where $\de > 0$ was chosen so that $V_\de \ge 0$, and then used the technique from \cite{KSW15} to produce a positive multiplier associated to $-\LP + V_\de$ .
Because $\de > 0$, our analysis of the resulting first-order equation had to be done in a vector-valued setting, thereby requiring a number of new ideas.
In the current paper, we avoid this lifting technique and instead modify the approach from \cite{KSW15}.
To prove Theorem \ref{LandisThm}, we directly construct positive super- and subsolutions associated to our equation under the assumption that $\gr a_{ij}$, $V_-$ and $W$ are sufficiently small in norm.

The second challenge involves using the local order of vanishing estimates to prove the unique continuation at infinity theorems.
Because of the decay conditions that we impose, a single application of the scaling technique from \cite{BK05} doesn't yield a very strong result.
Therefore, we need to design an iteration scheme and repeatedly apply an observation based on the scaling argument in \cite{BK05}.
The iteration scheme in this article is very similar to the one that we developed in \cite{DKW18} (which was based on the one in \cite{Dav14}), but because we are in a variable coefficient setting, we require quantitative control of the ellipticity and boundedness parameters.
This is one of the reasons why assumption \eqref{Lips} is so important.

The outline of this article is as follows.
To describe the main ideas in the proof, we begin with a preliminary section in which we describe the proof of Theorem \ref{LandisThm} for the case where $\mathcal{L} = - \LP + V$.
This sketch is presented in Section \ref{S2}.
In Section \ref{S3}, we show how to produce positive multipliers under different sets of conditions.
Then we prove a collection of bounds for these positive solutions.
Fundamental solutions and quasi-balls are introduced in Section \ref{quasi}.
While the first part of this section is reproduced from \cite{DKW17} and \cite{DW17}, the second part provides more refined bounds on the inner and outer radii of our quasi-balls.
The Beltrami operators are reviewed in Section \ref{Beltrami}.
Since we are able to make some simplifying assumptions for our setting, the results presented here are much simpler than those that previously appeared in \cite{DKW17} and \cite{DW17}.
The proof of Theorem \ref{OofV} is contained in Section \ref{ordVan}.
Section \ref{Landis0} presents an important proposition and uses it to prove Theorem \ref{LandisThm0}.
Finally, Section \ref{Landis} presents the iteration argument that leads to the proof of Theorem \ref{LandisThm}.

\section{The proof idea for $-\LP + V$}
\label{S2}

Given that the arguments for Theorems \ref{LandisThm0} and \ref{LandisThm} become technically complicated in the general setting, we present a sketch of the simplest case of Theorem \ref{LandisThm} when the operator is $\mathcal{L} = - \LP + V$.
That is, we will describe the proof of the following theorem:

\begin{thm0}[Landis-type theorem]
Assume that $V = V_+ - V_-$, where $V_+$ is bounded and $V_-$ exhibits polynomial decay at infinity. 
Let $u$ be a real-valued solution to $$-\LP u + V u = 0 \quad \text{ in } \R^2$$ that is polynomially bounded and normalized.
Then for any $\epsilon > 0$ and any $R >> 1$, we have
\begin{equation*}
\inf_{\abs{z_0} = R} \norm{u}_{L^\iny\pr{B_1\pr{z_0}}} \ge \exp\pr{- R^{1+\eps}}.
\end{equation*}
\end{thm0}

To prove this theorem, we use an iteration scheme.
The main observation is that we can improve (reduce) the exponent in the lower bound estimate by moving further and further from the origin.
The next proposition describes this iteration scheme.
Note that $\be < \al$ and $S < R$.

\begin{prop0}[Iteration proposition]
Assume that $V = V_+ - V_-$, where $V_+$ is bounded and $V_-$ exhibits polynomial decay at infinity. 
Let $u$ be a real-valued solution to $$-\LP u + V u = 0 \quad \text{ in } \R^2$$ that is polynomially bounded.
Assume that for $S >> 1$, there exists an $\al > 1$ so that 
\begin{align*}
\inf_{\abs{z_0} = S}\norm{u}_{L^\iny\pr{B_1\pr{z_0}}} \ge \exp\pr{- S^\al}.
\end{align*}
Then for any positive $\ga << 1$ and any $R \ge S + S^{1 + \ga} - 1 > S$, it holds that
\begin{equation*}
\inf_{\abs{z_1} = R} \norm{u}_{L^\iny\pr{B_1\pr{z_1}}} \ge \exp\pr{- C_2 R^\be \log R},
\end{equation*}
where $C_2$ is a universal constant and $\be = \max\set{\frac \al {1 + \ga}, 1} + \frac \ga {1 + \ga}$.
\end{prop0}

The idea behind the iterative argument that proves the Landis-type theorem is as follows.
To get started, we apply the quantitative unique continuation theorem of Bourgain and Kenig from \cite{BK05} to establish our initial estimate.
We can choose $\al$ to be any number greater than $\frac 4 3$ by making $S$ sufficiently large and taking $z_0 \in \R^2$ with $\abs{z_0} = S$.
Then we apply the iteration proposition to get an estimate near another point $z_1$, where $z_1$ is further from the origin than $z_0$.
By appropriately choosing $\ga$, we can ensure that $\be < \al$.
Then we repeat the argument by using the bound at $z_1$ as the starting point, then applying the iteration proposition to get a better bound at some further point $z_2$.
By carefully choosing the value of $\ga$ at each application of the proposition, we can ensure that after a finite number of turns, the value of $\be$ will be arbitrarily close to $1$, proving the theorem.

To prove the iteration proposition, we rely upon a scaling argument reminiscent of the one in \cite{BK05} in combination with an order of vanishing result.
The version of the order of vanishing statement that we use for $-\LP + V$ is as follows:

\begin{thm0}[Order of vanishing]
Let $F$ be a function such that $K^{-1} \lesssim F\pr{K} < 1$ for all $K \ge 1$.
For some $K \ge 1$, define $d = 1 - F\pr{K}$, $b = 1 + F\pr{K}$, and $m = 1 + 2F\pr{K}$.
Assume that $\norm{V}_{L^\iny\pr{B_m}} \le K^2$ and that there exists a positive function $\phi$ that solves $-\LP \phi + V \phi = 0$ in $B_m$.
Let $u$ be a real-valued solution to $-\LP u + V u = 0$ in $B_m$ that satisfies
\begin{align*}
& \norm{u}_{L^\iny\pr{B_{{m}}}} \le \exp\pr{C_1 K} \\
& \norm{u}_{L^\iny\pr{B_d}} \ge \exp\pr{- c_1 K^p},
\end{align*} 
for some $p \ge 0$.
Then for any $r$ sufficiently small, 
\begin{equation*}
\norm{u}_{L^\iny\pr{B_r}} \ge r^{C K^q /F\pr{K}},
\end{equation*}
where $q = \max\set{1, p}$ and $C = C\pr{\la, C_1, c_1}$.
\end{thm0}

This theorem, as presented here, very much resembles the order of vanishing estimate from \cite{KSW15} with two notable differences.
First, instead of working on balls for which the differences of radii is on the order of $1$, here we need that the differences between the radii of $B_b$, $B_d$ and $B_m$ are small.
This technical assumption ensures that we can run the iteration scheme.
Since we use information near $z_0$ to get an estimate near $z_1$, but we also don't want to get too close to the origin (because the estimates won't be very good anymore and we won't be able to construct a positive multiplier), we have to restrict the size of the radii.
The other difference between our order of vanishing theorem and the one from \cite{KSW15} is the lower bound in the hypothesis.
Instead of assuming that the norm on the small ball is bounded below by a constant, we assume that it is bounded below by an exponential that we quantify in terms of $K$.
Again, this is crucial to the technicalities of the iteration scheme.

We briefly describe the proof of the order of vanishing theorem.
For the solution $u$ and the positive solution $\phi$, we define $v = u/\phi$ and note that $\di\pr{\phi^2 \gr v} = 0$.
Associated to this divergence-free equation is a stream function, $\tilde v$.
It can be shown that function $f := \phi^2 v + i \tilde v$ is a solution $\bar \del f = \al f$, where $\al$ is a bounded function for which $\norm{\al}_{\iny} \lesssim K$.
An application of the similarity principle then shows that $f = e^\om h$, where $\norm{\om}_{\iny} \lesssim K$ and $h$ is holomorphic.
The next step is to apply the Hadamard three-circle theorem to $h$ on three balls with radii $\frac r 2 < d < b$.
In this setting, the power is $\te = \frac{\log\pr{b/d}}{\log\pr{2b/r}}$, and this explains why $F\pr{K}$ comes out in denominator of the exponent in the final estimate. 
We then use elliptic estimates and simplify to reach the conclusion of the theorem.

Now we return to the proof of the iteration proposition.
For some small $\ga > 0$, let $T = S^{1+\ga}$.
Choose some $z_1 \in \R^2$ so that $\abs{z_1} = R := S + T - 1$ and define $z_0 = S \frac{z_1}{\abs{z_1}}$.
Define $\tilde u\pr{z} = u\pr{z_1 + T z}$ and $\tilde V\pr{z} = {T}^2 V\pr{z_1 + T z}$.
Then with $m = 1 + \frac{S}{2T}$, $-\LP \tilde u + \tilde V \tilde u = 0$ in $B_m$.
Assuming for the time being that a positive solution exists in $B_m$, the order of vanishing theorem is applicable with $K = T$, $F\pr{K} = \frac{S}{4T} = \frac 1 4 T^{-\frac{\ga}{1+\ga}}$, and $p =\frac \al {1 + \ga}$.
Choosing $r = \frac 1 T$ and noting that $\norm{\tilde u}_{B_{r}\pr{0}} = \norm{u}_{B_1\pr{z_1}}$ leads to the conclusion of the proposition.

To complete the argument, we must justify that a positive solution exists.
Since the negative part of $V$ is assumed to decay polynomially at infinity, this means that if $\ga$ is sufficiently small compared to this rate of decay, then $\tilde V_-$ is bounded by some small fixed constant in $B_m$.
This observation implies that a function of the form $m^2 + 1 - x^2 - y^2$ is a positive supersolution in $B_m$. 
Since $e^{c T x}$ is a positive subsolution, and we can find $C$ so that $C\pr{m^2 + 1 - x^2 - y^2} \ge  e^{c T x}$, there must exist a positive solution, closing the remaining gap in the proof.

\section{Positive multipliers}
\label{S3}

The positive multiplier is a crucial ingredient in the proof of Theorem \ref{OofV}.
Our first lemma shows that the existence of a positive supersolution implies the existence of a positive solution.

\begin{lem}
\label{posMul0}
Recall that $0 < d < 1 < b < m$ are as defined in \eqref{dDef} -- \eqref{mDef} for some $K \ge 1$.
Let the coefficient matrix $A$ satisfy \eqref{ellip} and \eqref{Abd} in $B_m$.
Assume that $\norm{\gr a_{ij}}_{L^\iny\pr{B_m}} \le K$, $\norm{V}_{L^\iny\pr{B_m}} \le K^2$, and $\norm{W}_{L^\iny\pr{B_m}} \le K$.
If there exists a positive supersolution $\phi_2$, defined and continuous in $B_m$, then there exists a positive multiplier $\phi$, defined and continuous in $B_m$, that satisfies \eqref{ePDE}.
\end{lem}

\begin{proof}
We first construct a positive subsolution.
Set $\phi_1\pr{x,y} = \exp\pr{c K x}$ for some $c$ to be determined.  
A computation shows that
\begin{align*}
\mathcal{L} \phi_1 
&=- \di \pr{A \gr \phi_1} + W \cdot \gr \phi_1 + V \phi_1 
= \brac{-\pr{\del_x a_{11} + \del_y a_{21} + a_{11} cK} cK  + W_1 cK + V }\phi_1 \\
&\le \brac{\pr{ \norm{\gr a_{11}}_{L^\iny\pr{B_m}} + \norm{\gr a_{21}}_{L^\iny\pr{B_m}} - a_{11} cK} cK  + \norm{W}_{L^\iny\pr{B_m}} cK + \norm{V_+}_{L^\iny\pr{B_m}} }\phi_1 \\
&\le \brac{ \pr{2 K  - \la c K}c K + c K^2 + K^2} \phi_1
= - \brac{ c\pr{\la c - 3} - 1} K^2 \phi_1.
\end{align*}
If we choose $c$ sufficiently large with respect to $\la$, then $\mathcal{L}\phi_1 \le 0$ and $\phi_1$ is a subsolution.
Since $\phi_2$ is continuous in $B_m$, then it is bounded, so there exists $C > 0$ so that $C\phi_2 \ge \phi_1$ in $B_m$.
It follows that there exists a positive solution $\phi$ to \eqref{ePDE} in $B_m$.
\end{proof}

If we assume that $\gr a_{ij}$, $V_-$ and $W$ are bounded and small in norm, we can show the existence of a positive multiplier by directly constructing a positive subsolution and applying the previous lemma. 

\begin{lem}
\label{posMul}
Recall that $0 < d < 1 < b < m$ are as defined in \eqref{dDef} -- \eqref{mDef} for some $K \ge 1$.
Let the coefficient matrix $A$ satisfy \eqref{ellip} and \eqref{Abd} in $B_m$.
Assume that $\norm{\gr a_{ij}}_{L^\iny\pr{B_m}} \le \frac \la {4m}$, $\norm{V_+}_{L^\iny\pr{B_m}} \le K^2$, $\norm{V_-}_{L^\iny\pr{B_m}} \le \frac \la {m^2 + 1}$, and $\norm{W}_{L^\iny\pr{B_m}} \le \frac \la {2m}$.
Then there exists a positive multiplier $\phi$, defined and continuous in $B_m$, that solves \eqref{ePDE}. 
\end{lem}

\begin{proof}
Since $\frac \la {4m}, \sqrt{\frac \la {m^2 + 1}}, \frac \la {2m} \le K$ for some $K > 0$, then the first argument in Lemma \ref{posMul0} shows that there exists a positive subsolution $\phi_1$ with $\phi_1\pr{z} \le \exp\pr{c K}$ in $B_m$.
With $\phi_2 = \pr{m^2 + 1 - x^2 - y^2}$, we see that
\begin{align*}
\mathcal{L} \phi_2 
&= 2 \brac{\pr{\del_x a_{11}+ \del_y a_{21}}x + \pr{\del_x a_{12} + \del_y a_{22}} y + a_{11}+ a_{22}}  - 2 W \cdot \pr{x, y}  + V \pr{m^2 + 1 - x^2 - y^2} \\
&\ge 2\pr{a_{11} + a_{22}} 
- 2m \sum_{i,j=1,2}\norm{\gr a_{ij}}_{L^\iny\pr{Q_b}} 
- 2m \norm{W}_{L^\iny\pr{Q_b}}  
- \pr{m^2 + 1} \norm{V_-}_{L^\iny\pr{Q_b}} \\
&\ge 4\la - 8m \frac \la {4m} - 2m \frac \la {2m}  - \pr{m^2 + 1} \frac \la {m^2 + 1}
\ge 0.
\end{align*}
Therefore, $\phi_2$ is a supersolution.
Since $\exp\pr{c K} \phi_2 \ge \phi_1$ in $B_m$, then there exists a positive solution $\phi$ to \eqref{ePDE}.
\end{proof}

\begin{rem}
\label{rem2}
If $v$ is a solution to \eqref{ePDE} in $B_m$, then whenever $\al > 1$ and $\al r < m$, Theorem 8.32 from \cite{GT01}, for example, implies that $v$ satisfies the interior estimate
\begin{equation}
\norm{\gr v}_{L^\iny\pr{B_{r}}} \le \frac{C}{\pr{\al -1}r} \norm{v}_{L^\iny\pr{B_{\al r}}},
\label{intEst}
\end{equation}
where $C$ depends polynomially on $\la$ and $K$.
In particular, this estimate applies to both $u$ and $\phi$.
\end{rem}

Now we show that $\gr \pr{\log \phi}$ is bounded in $L^\iny$, an estimate that will be crucial to the arguments below.
We point out that this estimate holds for any positive solution that satisfies the listed hypotheses, not just those constructed in Lemma \ref{posMul}.

\begin{lem}
\label{grphiLem}
Recall that $1 < b < m$ are as defined in \eqref{bDef} and \eqref{mDef} for some $K \ge 1$.
Let $\phi$ be a positive solution to \eqref{ePDE} in $B_m$, where $A$ is a continuous matrix that satisfies \eqref{ellip} and \eqref{Abd}, $\norm{V}_{L^\iny\pr{B_m}} \le K^2$, and $\norm{W}_{L^\iny\pr{B_m}} \le K$.
Then there is a constant $C = C\pr{\la}$ for which
$$\norm{\gr \pr{\log \phi}}_{L^\iny\pr{B_b}} \le C K.$$
\end{lem}

\begin{proof}
Set $\vp = \frac{\log \phi}{C K}$ for some $C > 0$.
Then it follows from \eqref{ePDE} that in $B_m$ 
\begin{equation}
\mu \di\pr{A \gr \vp} + A \gr \vp \cdot \gr \vp = \widetilde W \cdot \gr \vp + \widetilde V,
\label{vpPDE}
\end{equation}
where $\mu = \frac{1}{C K}$, $\widetilde W = \frac{W}{CK}$ and $\widetilde V = \frac{V}{C^2 K^2}$.
The constant $C$ is chosen sufficiently large so that
\begin{align}
\mu < \frac {m - b} {4}, \quad
\norm{\widetilde W}_{L^\iny\pr{B_m}} \le 1, \quad 
\norm{\widetilde V}_{L^\iny\pr{B_m}} \le 1. 
\label{scaleBds}
\end{align}
The lower bound on $F\pr{K}$ and \eqref{mDef} ensure that the first condition may be satisfied.

\begin{clm}
For any $z \in B_b$, and $r \in \pr{\mu, \frac {m-b}{4}}$, if \eqref{vpPDE} and \eqref{scaleBds} hold, then there exists $C = C\pr{\la}$ such that
$$\int_{B_r\pr{z}} \abs{\gr \vp}^2 \le C r^2.$$
\label{clm1}
\end{clm}

\begin{proof}[Proof of Claim \ref{clm1}]
We use the abbreviated notation $B_r$ to denote $B_r\pr{z}$ for some $z \in B_b$.
Let $\eta \in C^\iny_0\pr{B_{2r}}$ be a cutoff function such that $\eta \equiv 1$ in $B_r$.
By the divergence theorem
\begin{align}
0 &= \mu \int \di \pr{A \gr \vp \, \eta^2} 
= \mu \int \di \pr{A \gr \vp} \eta^2 + 2\mu \int  A \gr \vp \cdot \gr \eta \eta.
\label{DivThmRes}
\end{align}
By \eqref{vpPDE} and \eqref{scaleBds},
\begin{align}
\int \mu \di \pr{A \gr \vp} \eta^2
&= - \int A \gr \vp \cdot \gr \vp \eta^2 + \int \widetilde W \cdot \gr \vp \eta^2 + \int \widetilde V \eta^2 \nonumber \\
&\le - \la \int \abs{\gr \vp}^2 \eta^2 
+ \frac \la 2 \int \abs{\gr \vp}^2 \eta^2
+ \frac 1 {2\la} \int \abs{\widetilde W}^2 \eta^2
+ \int \widetilde V \eta^2 \nonumber \\
&\le - \frac \la 2 \int \abs{\gr \vp}^2 \eta^2 + C r^2.
\label{IBd}
\end{align}
By Cauchy-Schwarz and Young's inequality,
\begin{align}
\abs{2 \mu \int \eta A \gr \vp \cdot \gr \eta} 
&\le 2 \mu \la^{-1} \pr{ \int \abs{\gr \vp}^2 \eta^2 }^{1/2} \pr{ \int \abs{\gr \eta}^2 }^{1/2} 
\le \frac{\la}{4} \int \abs{\gr \vp}^2 \eta^2 + C \mu^2.
\label{IIBd}
\end{align}
Combining \eqref{DivThmRes}-\eqref{IIBd} and using that $\mu < r$, we see that
\begin{align}
\int_{B_r} \abs{\gr \vp}^2 \le C \mu^2 + C r^2
\le C r^2,
\label{combinedEst}
\end{align}
proving the claim.
\end{proof}

We now use Claim \ref{clm1} to give a pointwise bound for $\gr \vp$ in $B_b$.
Define 
\begin{align*}
\vp_\mu\pr{z} = \frac{1}{\mu} \vp\pr{\mu z}, \qquad
A_\mu\pr{z} = A\pr{\mu z}, \qquad
L_\mu = \di A_\mu \gr.
\end{align*}
Then
\begin{align*}
&\gr \vp_\mu\pr{z} = \gr \vp \pr{\mu z} \\
&L_\mu \vp_\mu\pr{z} = \mu \di\pr{ A\pr{\mu z} \gr \vp\pr{\mu z}}.
\end{align*}
It follows from \eqref{vpPDE} that
\begin{align*}
L_\mu \vp_\mu\pr{z} + A_\mu \gr \vp_\mu \cdot \gr \vp_\mu
&= \mu \di\pr{ A\pr{\mu z} \gr \vp\pr{\mu z}}
+ A \pr{\mu z} \gr \vp \pr{\mu z} \cdot \gr \vp \pr{\mu z} \\
&= \widetilde V\pr{\mu z}
 + \widetilde W\pr{\mu z} \cdot \gr \vp\pr{\mu z}  \\
 &:= \widetilde V_\mu\pr{z}
 + \widetilde W_\mu\pr{z} \cdot \gr \vp_\mu\pr{z} ,
\end{align*}
where, from \eqref{scaleBds}, we have that
\begin{align*}
& \norm{\widetilde W_\mu}_{L^\iny\pr{B_1}} \le 1, \quad 
 \norm{\widetilde V_\mu}_{L^\iny\pr{B_1}} \le 1.
\end{align*}
Moreover,
\begin{align*}
\int_{B_{2}}\abs{\gr \vp\pr{\mu z}}^2 
= \frac{1}{\mu^2} \int_{B_{2 \mu}} \abs{\gr \vp}^2
\le \frac{1}{\mu^2} C\mu^2 = C,
\end{align*}
where we have used Claim \ref{clm1}.
Applications of Theorem 2.3 and Proposition 2.1 in Chapter V of \cite{Gi83} imply that there exists $p > 2$ such that
\begin{equation}
\norm{\gr \vp_\mu}_{L^p\pr{B_1}} \le C.
\label{QiaBd}
\end{equation}
Now we define
$$\tilde \vp_\mu \pr{z} = \vp_\mu\pr{z} - \frac{1}{\abs{B_1}} \int_{B_1} \vp_\mu.$$
Since $\gr \tilde \vp_\mu = \gr \vp_\mu$, then 
$$L_\mu \vp_\mu\pr{z} 
= - A_\mu \gr \vp_\mu \cdot \gr \vp_\mu + \widetilde V_\mu\pr{z} + \widetilde W_\mu\pr{z} \cdot \gr \vp_\mu\pr{z} : = \zeta \quad \text{ in } B_1.$$
Clearly, $\norm{\zeta}_{L^{p/2}\pr{B_1}} \le C$.
Moreover, by H\"older, Poincar\'e and \eqref{QiaBd},
$$\norm{\tilde \vp_\mu}_{L^{p/2}\pr{B_1}} 
\le C \norm{\tilde \vp_\mu}_{L^{p}\pr{B_1}} 
\le C \norm{\gr \tilde \vp_\mu}_{L^p\pr{B_1}} \le C.$$
By Theorem 9.9 from \cite{GT01}, for example, since $A$ is continuous,
$$\norm{\tilde \vp_\mu}_{W^{2, p/2}\pr{B_r}} \le C,$$
for any $r < 1$.
If $p > 4$, then it follows that $\norm{\gr \tilde \vp_\mu}_{L^\iny\pr{B_{r^\prime}}} \le C$.
Otherwise, assuming that $p < 4$, a Sobolev embedding shows that $\norm{\gr \tilde \vp_\mu}_{L^{\frac{2p}{4-p}}\pr{B_r}} \le C$.
Since $\frac{2p}{4-p} > p$, we may repeat these arguments to show that for some $r^\prime < 1$,
$$\norm{\gr \vp}_{L^\iny\pr{B_{ \mu r^\prime}}} =  \norm{\gr \vp_\mu}_{L^\iny\pr{B_{r^\prime}}} = \norm{\gr \tilde \vp_\mu}_{L^\iny\pr{B_{r^\prime}}} \le C.$$
This derivation works for any $z \in B_b$ and any $\mu < \mu_0$. 
Since $\vp = \frac{\log \phi}{C K}$, the conclusion of the lemma follows.
\end{proof}

We conclude this section by showing that all positive solutions to \eqref{ePDE} satisfy an important pointwise bound.

\begin{lem}
\label{phiBdLem}
Recall that $1 < b < m$ are as defined in \eqref{bDef} and \eqref{mDef} for some $K \ge 1$.
Let $\phi$ be a positive solution to \eqref{ePDE} in $B_m$, where $A$ is a continuous matrix that satisfies \eqref{ellip} and \eqref{Abd}, $\norm{V}_{L^\iny\pr{B_m}} \le K^2$, and $\norm{W}_{L^\iny\pr{B_m}} \le K$.
If $\phi$ is normalized so that $\phi\pr{0} = 1$, then 
\begin{equation}
\exp\pr{- c K} \le \phi\pr{z} \le \exp\pr{c K} \quad \text{ for a.e. } z \in B_b, 
\label{phiBounds}
\end{equation}
where $c = c\pr{\la, b}$.
\end{lem}

\begin{proof}
Assuming that $\phi\pr{0} = 1$, Lemma \ref{grphiLem} implies that for a.e. $z \in B_{b}$,
\begin{align*}
\abs{\log \phi\pr{z}} 
&= \abs{\log \phi\pr{z} - \log \phi\pr{0}} 
\le \norm{\gr \pr{\log \phi}}_{L^\iny\pr{B_{b}}} \abs{z}
\le C K b,
\end{align*}
and the claimed result follows.
\end{proof}

\section{Fundamental solutions and quasi-balls}
\label{quasi}

Given that we are working with variable-coefficient operators instead of the Laplacian, we no longer have the Hadamard three-circle theorem available to us as it was in \cite{KSW15}.
When we generalize the Hadamard three-circle theorem to the variable coefficient setting, we need to introduce so-called {\em quasi-balls}.
These sets were first introduced in \cite{DKW17}, and used again in \cite{DW17}.
Here we repeat a number of the definitions and facts that we previously produced.
We also use a perturbation argument to establish some new properties of quasi-balls associated to operators with Lipschitz continuous coefficients.

Let $L := - \di\pr{A \gr}$ be a second-order divergence form operator acting on $\R^2$ that satisfies the ellipticity and boundedness conditions described by \eqref{ellip} and \eqref{Abd}.
We start by discussing the fundamental solutions of $L$.
These results are based on the Appendix of \cite{KN85}.

\begin{defn}
A function $\Ga$ is called a {\em fundamental solution} for $L$ with pole at the origin if $\Ga \in H^{1,2}_{loc}\pr{\R^2 \setminus \set{0}}$, $\Ga \in H^{1,p}_{loc}\pr{\R^2}$ for all $p < 2$, and for every $\vp \in C^\iny_0\pr{\R^2}$
$$\int a_{ij}\pr{z} D_i \Ga\pr{z} \,  D_j \vp\pr{z} dz = \vp\pr{0}.$$
Moreover, $\abs{\Ga\pr{z} } \le C \log \abs{z}$, for some $C > 0$, {$|z|\ge C$}.
\end{defn}

\begin{lem}[Theorem A-2, \cite{KN85}]
There exists a unique fundamental solution $\Ga$ for $L$, with pole at the origin and with the property that $\disp \lim_{\abs{z} \to \iny} \Ga\pr{z} - g\pr{z} = 0$, where $g$ is a solution to $L g = 0$ in $\abs{z} > 1$ with $g = 0$ on $\abs{z} = 1$.
Moreover, there are constants $C_1, C_2, C_3, C_2, R_1 < 1 < R_2$ that depend on $\la$ such that
\begin{align*}
&C_1 \log\pr{\frac{1}{\abs{z}}} \le \Ga\pr{z} \le C_2 \log \pr{\frac{1}{\abs{z}}} \;\; \text{ for } \abs{z} < R_1 \\
& C_3 \log\abs{z} \le - \Ga\pr{z} \le C_2 \log \abs{z} \;\; \text{ for } \abs{z} > R_2.
\end{align*}
\label{fundSolBds}
\end{lem}

The level sets of $\Ga$ will be important to us.

\begin{defn}
\label{quasiDef}
Define a function $\ell: \R^2 \to \pr{0, \iny}$ as follows: $\ell\pr{z} = s$ iff $\Ga\pr{z} = - \frac 1 {2\pi} \log s$.
Then set 
\begin{align*}
Z_s &= \set{ z \in \R^2 : \Ga\pr{z} = - \frac 1 {2\pi}\log s} 
= \set{ z \in \R^2 : \ell\pr{z} = s} .
\end{align*}
We refer to these level sets of $\Ga$ as {\em quasi-circles.}
That is, $Z_s$ is the quasi-circle of radius $s$.
We also define (closed) {\em quasi-balls} as
\begin{align*}
Q_s &= \set{ z \in \R^2 : \ell\pr{z} \le s} .
\end{align*}
Open {\em quasi-balls} are defined analogously.
We may use the notation $Q_s^L$ and $Z_s^L$ to remind ourselves of the underlying operator.
\end{defn}

Although this is not the simplest way to define these sets (compare to our definitions in \cite{DKW17} and \cite{DW17}), here we introduced a scaling so that if $L = - \LP$, then $Q_s = B_s$.
This will be helpful below when derive more precise estimates for quasi-balls using perturbation arguments.

The following lemma follows from the bounds given in Lemma \ref{fundSolBds}.
The details of the proof may be found in \cite{DKW17}.

\begin{lem}
There are constants $c_1, c_2, c_3, c_4, c_5, c_6, S_1 < 1 < S_2$, that depend on $\la$, such that if $z \in Z_s$, then
\begin{align*}
& s^{c_1} \le \abs{z} \le s^{c_2}  \;\; \text{ for } s \le S_1 \\
& c_5 s^{c_1} \le \abs{z} \le c_6 s^{c_4} \;\; \text{ for } S_1 < s < S_2 \\
& s^{c_3} \le \abs{z} \le s^{c_4} \;\; \text{ for } s \ge S_2.
\end{align*}
\label{ZsBounds}
\end{lem}

Thus, the quasi-circle $Z_s$ is contained in an annulus whose inner and outer radii depend on $s$ and $\la$.
For future reference, it will be helpful to have a notation for the bounds on these inner and outer radii.
If we define
\begin{align*}
& \si\pr{s; L} = \sup\set{ r > 0 : B_r \su Q_s^L }, \quad \rho\pr{s; L} = \inf \set{r > 0 : Q_s^L \su B_r },
\end{align*}
then $\si\pr{s; L} \le \abs{z} \le \rho\pr{s;L}$ for all $z \in Z_s^L$.

\begin{rem}
Note that these definitions of $\si$ and $\rho$ differ from those that we introduced in \cite{DKW17} and \cite{DW17}.
In those papers, the bounds were established over a collection of operators with common ellipticity and boundedness conditions, whereas we are now defining them for a single operator.
That we are able to reduce from a class of operators to a single one is somewhat technical, but hinges on the simplifying assumptions on $A$.
\end{rem}

Since $\Ga = \Ga\pr{z} = \Ga\pr{z, 0}$ is defined to be a fundamental solution with a pole at the origin, the quasi-balls and quasi-circles just defined above are centered at the origin, so we may sometimes use the notation $Z_s\pr{0}$ and $Q_s\pr{0}$.
If we follow the same process for any point $z_0 \in \R^2$, we may discuss the fundamental solutions with pole at $z_0$, $\Ga\pr{z, z_0}$, and we may similarly define the quasi-circles and quasi-balls associated to these functions.
We denote the quasi-circle and quasi-ball of radius $s$ centered at $z_0$ by $Z_s\pr{z_0}$ and $Q_s\pr{z_0}$, respectively.

As we are working with operators whose coefficients are Lipschitz continuous, we can make a comparison between the quasi-balls of these operators and those of constant coefficient operators.
Let $A_0 = A\pr{0}$, a real, symmetric, and positive-definite matrix that is diagonalizable.
Thus, there exists an orthogonal matrix $O$ and a diagonal matrix $D$ so that 
$$A_0 = O^T D O = \brac{\begin{array}{cc} \cos \te & - \sin \te \\ \sin\te & \cos \te \end{array}} \brac{\begin{array}{cc} d_1 & 0 \\ 0 & d_2 \end{array}} \brac{\begin{array}{cc} \cos \te & \sin \te \\ -\sin\te & \cos \te \end{array}}.$$
The fundamental solution associated to the operator $L_0 := - \di\pr{A_0 \gr}$ is given by
\begin{equation}
\label{Ga0Def}
\Ga_0\pr{x, y} = -\frac 1 {2p} \log \brac{\frac{\pr{x \cos \te + y \sin \te}^2}{d_1} + \frac{\pr{- x \sin \te + y \cos \te}^2}{d_2}},
\end{equation}
where $p$ is the perimeter of the ellipse associated to $A_0$, and therefore depends on $d_1$ and $d_2$.

\begin{lem}
\label{FSClose}
Let the coefficient matrix $A$ satisfy \eqref{ellip} and \eqref{Abd}, and assume that $A$ is symmetric and Lipschitz continuous with $\abs{\gr a_{ij}\pr{z}} \le \de$ in $B_m$.
Let $\Ga$ denote the fundamental solution associated to $L := - \di\pr{A \gr}$, while $\Ga_0$ is the fundamental solution for $L_0$. 
Then there exists a constant $C = C\pr{\la, m}$ so that for any $z \in B_m$,
\begin{align*}
\abs{\Ga_0\pr{z} - \Ga\pr{z}} \le C \de.
\end{align*}
\end{lem}

\begin{proof}
Observe that
\begin{align*}
L\pr{\Ga_0 - \Ga}
&= -\di \brac{\pr{A - A_0} \gr \Ga_0} -\di \pr{A_0 \gr \Ga_0} + \di \pr{A \gr \Ga}
= -\di \brac{\pr{A - A_0} \gr \Ga_0}.
\end{align*}
Since $\abs{\gr \Ga_0\pr{z}} \le \frac{C}{\abs{z}}$ and $\abs{A\pr{z} - A_0} \le \de \abs{z}$, then $\norm{ \pr{A - A_0} \,\gr \Ga_0}_{L^\iny\pr{B_m}} \le C\de$.
It follows from a modification to the arguments in \cite{DW17} and \cite{DHM16} (see Definition 5 and Theorem 10 in \cite{DW17})  that
\begin{align*}
\Ga_0\pr{z} - \Ga\pr{z}
&= \int_{B_m} \gr_\zeta \Ga\pr{z, \zeta} \cdot \pr{A\pr{\zeta} - A_0} \, \gr \Ga_0\pr{\zeta} d\zeta
\end{align*}
and therefore
\begin{align*}
\sup_{z \in B_m} \abs{\Ga_0\pr{z} - \Ga\pr{z}}
&= \sup_{z \in B_m} \abs{\int_{B_m} \gr_\zeta \Ga\pr{z, \zeta} \cdot \pr{A\pr{\zeta} - A_0} \, \gr \Ga_0\pr{\zeta} d\zeta}.
\end{align*}
By H\"older's inequality, if $z \in B_m$, then
\begin{align*}
\abs{\int_{B_m} \gr_\zeta \Ga\pr{z, \zeta} \cdot \pr{A\pr{\zeta} - A_0} \, \gr \Ga_0\pr{\zeta} d\zeta}
&\le \norm{ \pr{A - A_0} \,\gr \Ga_0}_{L^\iny\pr{B_m}} \int_{B_m} \abs{D \Ga\pr{z, \zeta}} d\zeta \\
&\le C\de \norm{D \Ga\pr{z,\cdot}}_{L^1\pr{B_{2m}\pr{z}}}.
\end{align*}
It follows from a modification to the arguments that prove (B.18) from Theorem 10 in \cite{DW17} that $\norm{D \Ga\pr{z,\cdot}}_{L^1\pr{B_{2m}\pr{z}}} \le C$. 
The conclusion of the lemma follows.
\end{proof}

\begin{lem}
\label{rsClose}
Let the coefficient matrix $A$ satisfy \eqref{ellip} and \eqref{Abd}, and assume that $A$ is symmetric and Lipschitz continuous with $\abs{\gr a_{ij}\pr{z}} \le \de$ in $B_m$.
There exists a constant $C = C\pr{\la,m}$ so that for any $s > 0$, 
$$\rho\pr{s; L}\le \la^{-1/2} s^{p/2\pi} \exp\pr{C \de} \quad \text{and} \quad \si\pr{s; L} \ge \la^{1/2}s^{p/2\pi} \exp\pr{-C \de}.$$
\end{lem}

\begin{proof}
By definition, there exists $z_1, z_2 \in Z_s^{L}$ such that $\rho\pr{s; L}= \abs{z_1}$ and $\si\pr{s; L} = \abs{z_2}$.
Define $t_1, t_2 \in \R$ so that $z_1 \in Z_{t_1}^{L_0}$ and $z_2 \in Z_{t_2}^{L_0}$.
It follows from Definition \ref{quasiDef} and the definition of $\Ga_0$ given in \eqref{Ga0Def} that  for $i = 1, 2$, $\abs{z_i} \in \brac{d_2^{1/2} t_i^{p/2\pi}, d_1^{1/2} t_i^{p/2\pi}}$, where we have assumed that $d_1 \ge d_2$.
This observation also shows that $\rho\pr{s; L_0} = d_1^{1/2} s^{p/2\pi}$ and $\si\pr{s; L_0} = d_2^{1/2}s^{p/2\pi}$. 
As $z_1 \in Z_s^{L} \cap Z_{t_1}^{L_0}$, then Definition \ref{quasiDef} and Lemma \ref{FSClose} imply that 
\begin{align*}
\abs{\log \pr{ \frac {t_1} s }}
&= 2 \pi\abs{-\frac 1 {2\pi} \log s + \frac 1 {2\pi} \log t_1} 
= 2 \pi \abs{\Ga\pr{z_1} - \Ga_0\pr{z_1}} 
\le 2 \pi C \de.
\end{align*}
Similarly, since $z_2 \in Z_s^{L} \cap Z_{t_2}^{L_0}$, then $\abs{\log\pr{\frac s {t_2}}} \le 2 \pi C \de$.
Combining these observations shows that
\begin{align*}
\log \pr{\frac{\rho\pr{s; L}}{\rho\pr{s; L_0}}}
&\le \log \pr{\frac{ \abs{z_1}}{ t_1^{p/2\pi}}} + \frac p {2\pi}\abs{\log \pr{\frac{t_1} s}} - \log \pr{\frac{\rho\pr{s; L_0}}{s^{p/2\pi} }} 
\le C p \de
\end{align*}
and
\begin{align*}
\log \pr{\frac{\si\pr{s; L_0}}{\si\pr{s; L}}}
&\le \log \pr{ \frac{\si\pr{s; L_0}}{s^{p/2\pi}}} + \frac p {2\pi}\abs{\log \pr{\frac{s}{t_2}}} - \log \pr{\frac{\abs{z_2}}{t_2^{p/2\pi}}}
\le C p \de.
\end{align*}
To reach the conclusion, we combine these inequalities with the fact that $\la \le d_2 \le d_1 \le \la^{-1}$.
\end{proof}

\section{The Beltrami operators}
\label{Beltrami}

The aim of this section is to prove a Hadamard three-quasi-circle theorem and provide a similarity principle for solutions to Beltrami equations.
While much of this section is drawn from the work that was previously done in \cite{DKW17} and \cite{DW17}, we work from a simpler set of assumptions and therefore present modified versions of our previous results.
For all of the proofs of more general versions of these statements, we refer the reader to \cite{DKW17} and \cite{DW17}.

We assume throughout this section that $A$ is symmetric with determinant equal to $1$.
That is, $a_{12} = a_{21}$ and $\det A = 1$.
Associated to such an elliptic operator of the form $L = - \di A \gr$, we introduce a Beltrami operator that allows us to the reduce the second-order equation to a first-order system.
Define 
\begin{align}
D &=  \bar\del + \eta\pr{z} \del ,
\label{DDef} 
\end{align}
where $\bar \del = \tfrac{1}{2} \pr{\del_x + i \del_y}$, $\del = \tfrac{1}{2} \pr{\del_x - i \del_y}$, and
\begin{align*}
& \eta\pr{z} = \frac{a_{11} - a_{22} }{\det\pr{A + I}}  + i \frac{2 a_{12}}{\det\pr{A+I}}. 
\end{align*}
Since $\disp \abs{\eta\pr{z}}^2 = \frac{\tr A\pr{z} - 2}{\tr A\pr{z} + 2}$, then $\disp \abs{\eta\pr{z}} \le \sqrt{\frac{1-\la}{1 + \la}} < 1$.

Now we recall a pair of technical lemmas from \cite{DKW17} and refer the reader to that paper for their proofs.

\begin{lem}[Lemma 4.4 in \cite{DKW17}]
\label{decompLem}
Assume that $A$ is symmetric with determinant equal to $1$.
Then
$$\di A \gr = \pr{D + \widetilde W} \widetilde D,$$
where
\begin{align*}
\widetilde D &= \brac{1 + a_{11} - i a_{12}} \del_x + \brac{a_{12} - i\pr{1+a_{22}}} \del_y 
= \det\pr{A+I} \overline D  \\
\widetilde W 
&= \frac{\pr{\al \del_x a_{11} - \be \del_x a_{12} + \ga \del_y a_{11} + \de \del_y a_{12}} + i \pr{\ga \del_x a_{11} + \de \del_x a_{12} - \al \del_y a_{11} + \be \del_y a_{12}}}{ a_{11} \det\pr{A+I}^2} 
\end{align*}
\begin{align*}
&\al = a_{11} + a_{22} + 2 a_{11}a_{22} \qquad
\be = 2 a_{12}\pr{1 + a_{11}} \\
&\ga = a_{12}\pr{a_{22} - a_{11}} \qquad\qquad
\de = \pr{1 + a_{11}}^2 - a_{12}^2
\end{align*}
and $D$ is given by \eqref{DDef}.
\end{lem}

\begin{lem}[c.f. Lemma 7.1 in \cite{DKW17}]
\label{UpsLab}
For $\widetilde D$ as defined in the previous lemma, there exists $\Upsilon$ so that 
\begin{equation}
W \cdot \gr v =\Upsilon \widetilde D v.
\label{UpsEqn}
\end{equation}
Moreover, if $A$ satisfies \eqref{ellip} and \eqref{Abd}, $\norm{\gr a_{ij}}_{L^\iny\pr{B_{m}}} \le K$, and $\norm{W}_{L^\iny\pr{B_{m}}} \le K$, then there is a constant $C = C\pr{\la}$ so that
\begin{equation}
\norm{\Upsilon}_{L^\iny\pr{B_{m}}} \le C K.
\label{UpsBd}
\end{equation}
\end{lem}

The following is a simplified version of the Hadamard three-quasi-circle theorem.
The general version of this result appears in \cite[Theorem~4.5]{DKW17}, and was reproved in \cite[Corollary~2]{DW17}.
We refer the reader to these papers for the proof of the following result and other related ideas.
The quasi-balls in the following theorem are those related to the operator $L = - \di\pr{A \gr}$.

\begin{prop}[Hadamard three-quasi-circle theorem]
\label{3circle}
Let $f$ satisfy $Df=0$  in $Q_{s_0}$. 
Then for $0<s_1<s_2<s_3 < s_0$
\[
\norm{f}_{L^\iny\pr{Q_{s_2}}} \le \pr{\norm{f}_{L^\iny\pr{Q_{s_1}}}}^\te \pr{\norm{f}_{L^\iny\pr{Q_{s_3}}}}^{1 - \te},
\]
where
\[
\te=\frac{\log(s_3/s_2)}{\log(s_3/s_1)}.
\]
\end{prop}

Now we present the simplified similarity principle.
More general versions of these results appear in Section 4.4 of \cite{DKW17} and Section 4.2 of \cite{DW17}, and we refer the reader to those papers for the proofs. 
The approach is based on the work of Bojarksi, as presented in \cite{Boj09}. 
Recall the operators
$$T{\om}\pr{z} = -\frac{1}{\pi} \int_\Om \frac{\om\pr{\zeta}}{\zeta - z} d\zeta$$
$$S \om \pr{z} = -\frac{1}{\pi} \int_\Om \frac{\om\pr{\zeta}}{\pr{\zeta - z}^2} d\zeta,$$
where if $g \in L^p$ for some $p \ge 2$, then $T g$ exists everywhere as an absolutely convergent integral and $S g$ exists almost everywhere as a Cauchy principal limit.

\begin{prop}[see Theorems 4.1, 4.3 \cite{Boj09}]
\label{simPrinc}
Let $w$ be a generalized solution (possibly admitting isolated singularities) to
\begin{equation*}
Dw := \bar \del w + \eta\pr{z} \del w = A\pr{z} w
\end{equation*}
in a bounded domain $\Omega \su \R^2$.
Assume that $\abs{\eta\pr{z}} \le  k < 1$ in $\Om$, and $A$ belongs to $L^t\pr{\Om}$ for some $t > 2$.
Then $w\pr{z}$ is given by
$$w\pr{z}= f\pr{z} g\pr{z},$$
where $D f = 0$, $g\pr{z} = e^{\vp\pr{z}}$ with $\vp\pr{z} = T{\om}\pr{z}$, and $\om \in L^t\pr{\Om}$ is a solution to $\om + \eta S \om = A$.
Furthermore, 
$$\exp\pr{-C \norm{A}_{L^t\pr{\Omega}}} \le \abs{g\pr{z}} \le \exp\pr{C \norm{A}_{L^t\pr{\Omega}}}.$$
\end{prop}

\section{The proof of Theorem \ref{OofV}}
\label{ordVan}

To prove Theorem \ref{OofV}, we use the positive multiplier $\phi$ to reduce our PDE \eqref{ePDE} to a first-order Beltrami equation.
Then we apply the similarity principle and the Hadamard three-quasi-circle theorem to get a three-ball inequality for our solution function.

Let $u$ and $\phi > 0$ be as given, where $\phi$ has been normalized so that $\phi\pr{0} = 1$.
Define $\disp v = \frac u \phi$ and notice that since both functions are solutions to \eqref{ePDE} and there is no loss in assuming that $A = A^T$ (see Remark \ref{symm}), then
\begin{align*}
-\di\pr{A \gr v} 
&+ \pr{W - 2 A \gr {\log \phi}} \cdot \gr v = 0.
\end{align*}
Since we will rely upon the tools developed in Section \ref{Beltrami}, we need to reduce this equation to one where $A$ also has determinant equal to $1$.
A computation shows that
\begin{align}
\label{tePDE}
- \di\pr{\overline A \gr v}
+ \overline W \cdot \gr v
&= 0,
\end{align}
where
\begin{align*}
\overline{A} = \frac{ A}{\sqrt{\mbox{det} A}}, \qquad 
\overline W &= \frac{W - 2 A \gr \pr{\log \phi} }{\sqrt{\mbox{det} A}} + A \gr\pr{\frac{1}{\sqrt{\det  A}}}.
\end{align*}
The radii $d$, $b$, and $m$ are defined as in \eqref{dDef} -- \eqref{mDef} with respect to the operator $\overline L := - \di\pr{\overline A \gr}$.
Since $\overline A$ is symmetric with determinant equal to $1$, Lemmas \ref{decompLem} and \ref{UpsLab} imply that equation \eqref{tePDE} is equivalent to
\begin{align*}
-\pr{D + \widetilde W} \widetilde D v + \Upsilon \widetilde D v = 0,
\end{align*}
where $D$, $\widetilde D$ and $\widetilde W$ are now defined with respect to $\overline A$ and $\Upsilon$ depends on $\overline W$.
Upon setting $w = \widetilde D v$,
this equation reduces to
$$D w = \pr{\Upsilon - \widetilde W} w.$$
An application of the similarity principle described by Proposition \ref{simPrinc} shows that
$$w\pr{z} = f\pr{z} g\pr{z},$$
where $Df = 0$ in $B_b$.
Since $\norm{\widetilde W}_{L^\iny\pr{B_b}} \le CK$ and Lemma \ref{grphiLem} implies that $\norm{\overline W}_{L^\iny\pr{B_b}} \le CK$, then Lemma \ref{UpsLab} implies that $\norm{\Upsilon - \widetilde W}_{L^\iny\pr{B_b}} \le CK$ and it follows that for a.e. $z \in B_b$,
$$\exp\pr{-C K} \le \abs{g\pr{z}} \le \exp\pr{C K}.$$
We now apply Proposition \ref{3circle}, the three-quasi-circle theorem, to $f = g^{-1} w$ and use the bound on $g$ to get
\begin{align*}
\norm{w}_{L^\iny\pr{Q_{1}}}
&= \norm{g f}_{L^\iny\pr{Q_{1}}}
\le \exp\pr{CK} \norm{f}_{L^\iny\pr{Q_{1}}} 
\le \exp\pr{CK} \pr{\norm{f}_{L^\iny\pr{Q_{s/2}}}}^\te \pr{\norm{f}_{L^\iny\pr{Q_{s_2}}}}^{1 - \te},
\end{align*}
where $s << 1 < s_2 := 1 + F\pr{K}$ and $\te=\frac{\log(s_2)}{\log(2s_2/s)}$.
Let $r = 2\rho\pr{s/2}$ so that $Q_{s/2} \su B_{r/2}$.
It follows from \eqref{bDef} and the definition of $\rho$ that $Q_{s_2} \su B_{b}$.
Then we have
\begin{align*}
c \norm{\gr v}_{L^\iny\pr{Q_{1}}}
\le \norm{w}_{L^\iny\pr{Q_{1}}}
&\le \exp\pr{CK} \pr{\norm{w}_{L^\iny\pr{B_{r/2}}}}^\te \pr{\norm{w}_{L^\iny\pr{B_{b}}}}^{1 - \te} \\
&\le \exp\pr{CK} \pr{\norm{\gr v}_{L^\iny\pr{B_{r/2}}}}^\te \pr{\norm{\gr v}_{L^\iny\pr{B_{b}}}}^{1 - \te},
\end{align*}
where we have used \eqref{ellip} and \eqref{Abd} to conclude that $\abs{w} \sim \abs{\gr v}$.
The definition of $v$, Remark \ref{rem2}, Lemma \ref{grphiLem}, and Lemma \ref{phiBdLem} imply that
\begin{align*}
\norm{\gr v}_{L^\iny\pr{B_{r/2}}}
&\le \norm{\phi^{-1} \gr u}_{L^\iny\pr{B_{r/2}}} + \norm{\phi^{-1} u \gr \pr{\log\phi}}_{L^\iny\pr{B_{r/2}}} \\
&\le  \norm{\phi^{-1}}_{L^\iny\pr{B_{r/2}}}\brac{\frac{C\pr{\la, K}}{r} \norm{u}_{L^\iny\pr{B_{r}}}  + C K \norm{u}_{L^\iny\pr{B_{r/2}}} } 
\le \frac{\exp\pr{C K}}{r} \norm{u}_{L^\iny\pr{B_{r}}} .
\end{align*}
Similarly,
\begin{align*}
\norm{\gr v}_{L^\iny\pr{B_{b}}}
&\le \exp\pr{C K} \frac{C\pr{\la, K}}{m-b}  \norm{u}_{L^\iny\pr{B_{m}}}
\le \exp\brac{\pr{C+C_1}K} ,
\end{align*}
where we have applied \eqref{localBd} and that $\frac{1}{m - b} \lesssim K$.
Combining what we have so far shows that
\begin{align}
\norm{\gr v}_{L^\iny\pr{Q_{1}}}
&\le \exp\pr{CK} \pr{\frac {\norm{u}_{L^\iny\pr{B_{r}}}}{r}}^\te.
\label{3ball}
\end{align}

Towards completing the proof, we need to bound the lefthandside from below using the assumption from \eqref{localNorm} that $\norm{u}_{L^\iny\pr{B_d}} \ge \exp\pr{- c_1 K^p}$.
We repeat the argument from \cite{KSW15} here.
This assumption implies that there exists $z_0 \in B_d$ such that $\abs{u\pr{z_0}} \ge \exp\pr{- c_1 K^p}$. 
Without loss of generality, we'll assume that $u\pr{z_0} \ge \exp\pr{- c_1 K^p}$.
Since $u$ is real-valued, then for any $a > 0$, we have that either $u\pr{z} \ge a$ for all $z \in B_d$, or there exists $z_1 \in B_d$ such that $u\pr{z_1} < a$.
If the second case holds with $a = \frac{1}{2} \exp\pr{- 2 c K- c_1 K^p}$, then by \eqref{phiBounds} we see that
\begin{align*}
\frac{u\pr{z_1}}{\phi\pr{z_1}} 
\le \frac{a}{\phi\pr{z_1}} 
\le a \exp\pr{c K}
\le \frac{1}{2}\exp\pr{-c K- c_1 K^p},
\end{align*}
while
\begin{align*}
\frac{u\pr{z_0}}{\phi\pr{z_0}} 
\ge \exp\pr{- c_1 K^p - c K}.
\end{align*}
Since $B_d \su Q_{1 - F\pr{K}} \su Q_1$, it follows that
\begin{align*}
C \norm{\gr v}_{L^\iny\pr{Q_{1}}} 
\ge \abs{v\pr{z_0} - v\pr{z_1}} 
\ge \frac{u\pr{z_0}}{\phi\pr{z_0}} - \frac{u\pr{z_1}}{\phi\pr{z_1}}
\ge \frac{1}{2}\exp\pr{-c K- c_1 K^p}.
\end{align*}
Combining this bound with estimate in \eqref{3ball} shows that
\begin{align*}
\norm{u}_{L^\iny\pr{B_{r}}} 
&\ge r \exp\pr{-\frac{C K + c_1 K^p}\te}
\end{align*}
Recalling that $\te=\frac{\log(s_2)}{\log(2s_2/s)}$, we see that
\begin{align*}
- \frac 1 \te &= \frac{\log\pr{s/2} - \log s_2}{\log s_2}
= \frac{\log\pr{s/2}- \log \pr{1 + F\pr{K}}}{\log \pr{1 + F\pr{K}}} 
\ge \frac{C \log\pr{s/2}}{F\pr{K}}.
\end{align*}
An application of Lemma \ref{ZsBounds} shows that $c \log r \le \log \pr{r/2} \le c_2 \log\pr{s/2}$, so that
\begin{align*}
- \frac 1 \te &
\ge \frac{C \log\pr{r}}{F\pr{K}}
\end{align*}
and the conclusion of the Theorem \ref{OofV} follows.
On the other hand, if $u\pr{z} \ge a$ for all $z \in B_d$, then the conclusion is obviously satisfied.

\section{The proof of Theorem \ref{LandisThm0}}
\label{Landis0}

We begin with a proposition that serves as the main tool in the iteration scheme.

\begin{prop}
\label{itProp0}
Let the coefficient matrix $A$ be symmetric and satisfy \eqref{ellip}, \eqref{Abd}, and \eqref{Lips}.
Assume that $\norm{V}_{L^\iny\pr{\R^2}} \le \mu_1^2$, $\norm{W}_{L^\iny\pr{\R^2}} \le \mu_2$, and that $\mathcal{L}$ is non-negative in $\R^2 \setminus B_{S_0}$ for some $S_0 > 0$.
Let $u: \R^2 \to \R$ be a solution to \eqref{ePDE} for which \eqref{uBd} holds.
Suppose that for any $\ga \in\pr{0, \eps_0}$ and any $S \ge \tilde S\pr{\mu_0, \eps_0, C_0, c_0, S_0, \ga, \La}$, there exists an $\al \in \pb{1, 3}$ so that
\begin{align}
\inf_{\abs{z_0} = S}\norm{u}_{L^\iny\pr{B_\La\pr{z_0}}} \ge \exp\pr{- S^\al}.
\label{stepnEst}
\end{align}
Assume further that when we restrict to $z \in \R^2$ with $\abs{z} \in\brac{ \frac S 2,  4 S^{1 + \ga}}$, \eqref{ellip} and \eqref{Abd} hold with $\la \ge 1 - \frac {S^{-\ga}} {20}$.
Then for any $R \ge S + S^{1 + \ga} - \La$, it holds that
\begin{equation}
\inf_{\abs{z_1} = R} \norm{u}_{L^\iny\pr{B_\La\pr{z_1}}} \ge \exp\pr{- C_2 R^\be \log R},
\label{stepn1Est}
\end{equation}
where $C_2 = C_2\pr{\mu_1,\mu_2}$ and $\be = \max\set{\frac \al {1 + \ga}, 1} + \frac \ga {1 + \ga}$.
\end{prop}

\begin{proof}
Fix $\ga \in \pr{0, \eps_0}$ and $S \ge \tilde S$.
Define $T = S^{1 + \ga}$.
Let $z_1 \in \R^2$ be such that $\abs{z_1} = S + T - \La= R$ and define $z_0 = S \frac{z_1}{\abs{z_1}}$.

With $a = \pr{1 - \frac{S}{5T}}^{-1}$, define $\tilde u\pr{z} = u\pr{z_1 + aT z}$, $\widetilde A\pr{z} = A\pr{z_1 + aT z}$, $\widetilde W\pr{z} = aT W\pr{z_1 + aT z}$, and $\widetilde V\pr{z} = \pr{aT}^2 V\pr{z_1 + aT z}$ so that
\begin{align*}
- \di \pr{\widetilde A\pr{z} \gr \tilde u\pr{z} }
&+ \widetilde W\pr{z} \cdot \gr \tilde u \pr{z} + \widetilde V\pr{z} \tilde u\pr{z} = 0 .
\end{align*}
Define $d = \si\pr{1 - \frac S {20 T}}$, $b = \rho\pr{1 + \frac S {20 T}}$, and $m  = b + \frac {S}{20 T}$. 
The conditions on $A$, $V$, and $W$ imply that $\norm{\gr \tilde a_{ij}}_{L^\iny\pr{B_m}} \le \mu_0 a T$, $\norm{\widetilde V}_{L^\iny\pr{B_m}} \le \pr{\mu_1 a T}^2$, and $\norm{\widetilde W}_{L^\iny\pr{B_m}} \le \mu_2 aT$.
For Theorem \ref{OofV} to be applicable, we require a positive multiplier, $\phi$, defined in $B_m$.
Since $\mathcal{L}$ is assumed to be non-negative in $\R^2 \setminus B_{S_0}$, then there exists a positive supersolution in $\R^2 \setminus B_{S_0}$, and as long as $B_{aTm}\pr{z_1} \su \R^2 \setminus B_{S_0}$, an application of Lemma \ref{posMul0} implies that such a function $\phi$ exists in $B_m$.
For this set containment to hold, the following paragraph shows that it suffices to take $S \ge 2 S_0$.

Assuming that $a T m \le T + \frac {S}{2}- \La$, condition \eqref{Lips} implies that 
\begin{equation}
\label{grtA}
\abs{\gr \tilde a_{ij}\pr{z}} \le aT \mu_0 \pr{S/2}^{-\pr{1 + \eps_0}} \le \frac 5 4 \mu_0 2^{-\pr{1 + \eps_0}} S^{\ga-\eps_0} \quad \text{ in } B_m.
\end{equation}
An application of Lemma \ref{rsClose} then shows that
\begin{align*}
b \le \pr{1 - \frac {S} {20T}}^{- \frac 1 2}\pr{1 + \frac {S} {20 T}}^{\frac p {2\pi}} \exp\pr{5 C\mu_0 2^{-\pr{3 + \eps_0}} S^{\ga-\eps_0} } \\
d \ge  \pr{1 - \frac {S} {20T}}^{\frac 1 2}\pr{1 - \frac {S} {20 T}}^{\frac p {2\pi}} \exp\pr{-5 C \mu_0 2^{-\pr{3 + \eps_0}} S^{\ga-\eps_0} },
\end{align*}
where $C$ is a universal constant.
We choose $S$ sufficiently large with respect to $\mu_0$ and $\ga - \eps_0$ so that $b \le 1 + \frac{S}{5T}$ and $d \ge 1 - \frac S{5T}$.
Then $aT b \le T + \frac{9S}{20} -\La$ and, consequently, $a T m \le T + \frac{S}{2}- \La$, as required.
Condition \eqref{uBd} in combination with the upper bound on $aT m$ implies that 
\begin{align}
\norm{\tilde u}_{L^\iny\pr{B_m}} 
&= \norm{u}_{L^\iny\pr{B_{aT m}\pr{z_1}}} 
\le \norm{u}_{L^\iny\pr{B_{2T + \frac {3S}2 }\pr{0}}}
\le C_0 \pr{4T}^{c_0}
\le \exp\pr{T},
\label{tuUpper}
\end{align}
if $S$ is sufficiently large with respect to $C_0$ and $c_0$, while \eqref{stepnEst} with $aTd \ge T$ shows that
\begin{align}
\norm{\tilde u}_{L^\iny\pr{B_d}} 
&=\norm{u}_{L^\iny\pr{B_{aTd}\pr{z_1}}} 
\ge \norm{u}_{L^\iny\pr{B_{\La}\pr{z_0}}} 
\ge \exp\pr{- S^\al}.
\label{tuLower}
\end{align}
Set $a_0 = \frac 5 4 \max\set{\mu_1, \mu_2}$.
Then Theorem \ref{OofV} is applicable with $K = a_0 T$, $F\pr{K} = \frac S {20 T} = \frac{1}{20 } \pr{\frac {K}{a_0}}^{-\frac{\ga}{1 + \ga}}$, $C_1 = \frac 1 {a_0}$, $c_1 = \pr{\frac 1 {a_0}}^3$, and $p = \frac \al {1 + \ga} \le 3$.
That is, with $q = \max\set{1, p}$ and a constant $C$ depending on $a_0$, we have
\begin{align*}
\norm{u}_{B_{a T r}\pr{z_1}} 
\ge \norm{\tilde u}_{B_r\pr{z_1}} 
\ge r^{C K^q/ F\pr{K}}
\ge r^{20 \tilde C T^{q + \frac \ga {1 + \ga}}},
\end{align*}
where $\tilde C = C a_0^3$.
Setting $r = \frac {4\La}{5 T}$ shows that
\begin{align*}
\norm{u}_{B_{\La}\pr{z_1}} 
&\ge \exp\brac{-21 \tilde C T^{\be} \log T },
\end{align*}
where we have assumed that $\log T \ge 20 \log\pr{5/4\La}$.
Since $z_1 \in \R^2$ with $\abs{z_1} = R$ was arbitrary, the conclusion of the proposition follows.
\end{proof}

We now have everything required to prove Theorem \ref{LandisThm0}.

\begin{proof}[The proof of Theorem \ref{LandisThm0}]
We first consider the case where $\Om \ne \R^2$.
Let $\eps > 0$ be given. 
Fix $R \ge R_0$, where $R_0$ will be specified below.

As pointed out in Remark \ref{symm}, there is no loss in assuming that $A$ is symmetric.
We apply a change of variables so that for some $\hat z$ sufficiently far from the origin, the coefficient matrix is equal to the identity.
Since $A$ is real, symmetric and elliptic, then there exists a constant symmetric matrix $Q$ for which $Q^2 = A\pr{\la^{-1/2} R \vec e_1}$.
Define $\hat z = \la^{-1/2} R Q^{-1} \vec e_1$, and note that since $\norm {Q^{-1}} \ge \la^{1/2}$, then $\abs{\hat z} \ge R$.
Let $\tilde u\pr{z} = u\pr{Q z}$, $\widetilde A\pr{z} = Q^{-1}A\pr{Q z} Q^{-1}$, $\widetilde W\pr{z} = W\pr{Q z}Q^{-1}$, and $\widetilde V\pr{z} = V\pr{Q z}$ so that $\widetilde A\pr{\hat z} = I$ and
$$-\di \pr{\widetilde A \gr \tilde u}  + \widetilde W \cdot \gr \tilde u + \widetilde V \tilde u = 0.$$
Moreover, all of the hypotheses are satisfied where the new constants $\la$, $\mu_0$, $\mu_2$, and $C_0$ depend additionally on $\la$.
The condition that $\widetilde A\pr{\hat z} = I$, in combination with \eqref{Lips}, implies that if $\abs{z}$ is sufficiently large, then $\la$ is close to $1$.
We will make this statement rigorous below.

To start the iteration scheme, we apply Theorem 1.1 from \cite{LW14}, and therefore need to transform the elliptic equation \eqref{ePDE} into non-divergence form.
Notice that
\begin{align}
-\di \pr{\widetilde A \gr \tilde u} + \widetilde W \cdot \gr \tilde u + \widetilde V \tilde u
&=-  \sum_{i, j =1}^2 \tilde a_{ij}  \del_{ji} \tilde u + \pr{\widetilde W - W_n} \cdot \gr \tilde u + \widetilde V \tilde u,
\label{ndForm}
\end{align}
where $W_n := \pr{\del_x \tilde a_{11} + \del_y \tilde a_{21}, \del_x \tilde a_{12} + \del_y \tilde a_{22}}$ and condition \eqref{Lips} implies that $\widetilde W - W_n$ belongs to $L^\iny$.
As Theorem 1.1 from \cite{LW14} is applicable to the transformed equation, we conclude that there exists $\hat S_0 = \hat S_0\pr{C_0,\mu_0,\mu_1,\mu_2,\eps_0}$ so that whenever $S \ge \hat S_0$, it holds that
\begin{align*}
\inf_{\abs{z_0} = S}\norm{\tilde u}_{L^\iny\pr{B_{\sqrt \la}\pr{z_0}}} \ge \exp\pr{- C_3 S^{2} \pr{\log S}^{\eta \pr{S}}},
\end{align*}
where $\disp \eta\pr{S} = C' \pr{\log S} \pr{\log \log \log S} \pr{\log \log S}^{-2}$ and the constants $C_3$ and $C'$ both depend on $c_0$, $C_0$, $\mu_0$, $\mu_1$, $\mu_2$ and $\eps_0$.

Let $\ga := \min\set{\eps, \frac {\eps_0}{2}}$ and choose 
$$S_0 \ge \max\set{ \tilde S\pr{\mu_0, \eps_0, C_0, c_0, d\pr{\Om}, \ga, \sqrt \la}, \hat S_0, \pr{20 C_4}^{\frac 2{\eps_0}}},$$ 
where $\tilde S$ is as given in Proposition \ref{itProp0}, $d\pr{\Om} = \inf\set{ t : B_t \supset \R^2 \setminus \Om}$, and $C_4$ will be specified below.
If need be, choose $S_0$ even larger so that
\begin{align}
&S^{\ga^2/2} \ge \max\set{C_2 \log S, C_3 \pr{\log S}^{\eta \pr{S}}} \quad \text{ for all } S \ge S_0,
\label{logLarge} 
\end{align}
where $C_2$ is the constant from Proposition \ref{itProp0}. 
Consequently, for every $S \ge S_0$,
\begin{align}
\inf_{\abs{z_0} = S}\norm{\tilde u}_{L^\iny\pr{B_{\sqrt \la}\pr{z_0}}} \ge \exp\pr{- S^{\al_0}},
\label{step0}
\end{align}
where we have defined $\al_0 = 2 + \frac{\ga^2}{2}$.

We now quantify $\la$ to ensure that Proposition \ref{itProp0} is applicable.
Let $z \in \R^2$ be an arbitrary point for which $\abs{z} \in \brac{S_0/2, 4 R}$.
Define $\check z = \abs{\hat z} \frac{z}{\abs{z}}$ and let $\Ga$ denote the shortest path between $\hat z$ and $\check z$ that stays on the sphere of radius $\abs{\hat z}$.
Recalling that $A\pr{\hat z} = I$, we see that
\begin{align*}
\abs{\tilde a_{ij}\pr{z} - \de_{ij}} 
&\le \abs{\tilde a_{ij}\pr{z} - \tilde a_{ij}\pr{\check z}} + \abs{\tilde a_{ij}\pr{\check z} - \tilde a_{ij}\pr{\hat z}} \\
&\le \int_{\abs{z}}^{\abs{\hat z}} \abs{ \gr \tilde a_{ij}\pr{t \frac{z}{\abs{z}}}} dt  + \sup_{z \in \Ga} \abs{\gr \tilde a_{ij}\pr{z}} \ell\pr{\Ga} \\
&\le  \int_{\abs{z}}^{\abs{\hat z}} \mu_0 t^{-\pr{1 + \eps_0}} dt + \mu_0 \abs{\hat z}^{-\pr{1 + \eps_0}} \pi \abs{\hat z}
\le \mu_0 \pr{\eps_0^{-1} + \pi} \abs{z}^{- \eps_0},
\end{align*}
where we have assumed that $\abs{z} \le \abs{\hat z}$.
If $\abs{z} \ge \abs{\hat z} \ge R$, then an analogous estimate shows that $\abs{\tilde a_{ij}\pr{z} - \de_{ij}}  \le C \abs{\hat z}^{-\eps_0} \le C \abs{z}^{-\eps_0}$ since $\abs{z} \le 4 R$.
Taking $\abs{z} \ge S/2$ for some $S \ge S_0$, it follows from this bound that there is another constant $C_4$ so that over this region, $\la \ge 1- C_4 S^{-\eps_0}$.
As $\eps_0- \ga \ge \frac{\eps_0} 2$, then $S^{\eps_0 -\ga} \ge S_0^{\eps_0/2} \ge 20 C_4$ and we see that $C_4 S^{-\eps_0} \le \frac{S^{- \ga}}{20}$.
Therefore, we are in a position to apply Proposition \ref{itProp0}.

If $\al_0 \le 1 + \ga \le 1 + \eps$, then there is no need to iterate so we define $N = -1$.
Otherwise, we assume that $\al_0 > 1 + \ga$.
With $S_1 = S_0 + S_0^{1 + \ga} - \sqrt \la$ and $\be_0 = \frac{\al_0+ \ga}{1 + \ga} $, Proposition \ref{itProp0} with \eqref{step0} and $\La = \sqrt \la$ implies that
\begin{align*}
\inf_{\abs{z_1} = S_1}\norm{\tilde u}_{L^\iny\pr{B_{\sqrt \la}\pr{z_1}}} 
&\ge \exp\pr{- C_2 S_1^{\be_0} \log S_1}
\ge \exp\pr{- S_1^{\al_1}},
\end{align*}
where in the second inequality we have applied \eqref{logLarge} and defined $\al_1 = \be_0 + \frac{\ga^2}2$.
We iterate this argument by defining $S_{n+1} = S_n + S_n^{1 + \ga} - \sqrt \la$ and $\al_{n+1} = \frac{\al_n + \ga}{1 + \ga} + \frac{\ga^2}{2}$ for $n = 1, \ldots, N$, where $N \in \N$ is defined so that $\al_N > 1 + \ga$ and $\al_{N+1} \le 1 + \ga$.
Repeated applications of Proposition \ref{itProp0} then show that for all $n = 1, \ldots, N+1$,
\begin{align*}
\inf_{\abs{z_n} = S_n}\norm{\tilde u}_{L^\iny\pr{B_{\sqrt \la}\pr{z_n}}} 
&\ge \exp\pr{- S_n^{\al_n}}.
\end{align*}
Whenever $\al_n \ge 1 + \ga$, it can be shown that $\al_{n+1} \le \pr{1 - \frac{\ga^2}2} \al_n$ and it follows that 
$$N \le N_0 := \ceil*{ \log\pr{\frac{1+\ga}{\al_0}}/\log\pr{1 - \frac{\ga^2}2}} - 1.$$
Therefore, the iteration stops after a finite number of steps.

The final step is to undo the change of variables that we introduced at the beginning of the proof.
Recall that $\tilde u \pr{z} = u\pr{Qz}$.
 Since $Qz \in B_{\sqrt \la}\pr{Q z_0}$ implies that $z \in B_1\pr{z_0}$ and $\abs{z_0} \ge \la^{-1/2} S_{N+1}$ implies that $\abs{Q z_0} \ge S_{N+1}$, then for any $S \ge S_{N+1}$,
\begin{align*}
\inf_{\abs{z_0} = \la^{-1/2}S}\norm{u}_{L^\iny\pr{B_{1}\pr{z_0}}} 
&\ge \inf_{\abs{Q z_0} = S} \norm{\tilde u}_{L^\iny\pr{B_{\sqrt \la}\pr{Q z_0}}} 
\ge \exp\pr{- S^{\al_{N+1}}}
\ge \exp\pr{- S^{1 + \eps}},
\end{align*}
where we have used that $\ga \le \eps$.
If we define $R_0 := \la^{-1/2} S_{N+1}$, then the conclusion of the theorem follows.

We now consider the case where $\Om = \R^2$.
Choose $R \ge R_0$, where $R_0$ will be specified below. 
Let $z_0 \in \R^2$ be such that $\abs{z_0} = R $.
Define $d = \si\pr{\frac 4 5}$, $b = \rho\pr{\frac 6 5}$, and $m = b + \frac 1 5$, then set $\tilde u\pr{z} = u\pr{z_0 + d^{-1} R z}$, $\widetilde A\pr{z} = A\pr{z_0 + d^{-1} R z}$, $\widetilde W\pr{z} = d^{-1} R \, W\pr{z_0 + d^{-1} R z}$, and $\widetilde V\pr{z} = \pr{d^{-1}R}^2 V\pr{z_0 + d^{-1} R z}$ so that
\begin{align*}
- \di \pr{\widetilde A\pr{z} \gr \tilde u\pr{z} }
&+ \widetilde W\pr{z} \cdot \gr \tilde u \pr{z} + \widetilde V\pr{z} \tilde u\pr{z} = 0 .
\end{align*}
With $\hat \mu = \max\set{\mu_i : i = 0, 1, 2}$, set $K = \hat \mu d^{-1} R$ and we see that $\norm{\gr \tilde a_{ij}}_{L^\iny\pr{B_m}} \le K$, $\norm{\widetilde V}_{L^\iny\pr{B_m}} \le K^2$, and $\norm{\widetilde W}_{L^\iny\pr{B_m}} \le K$.
Since $\mathcal{L}$ is assumed to be non-negative in $\R^2$, there is a positive supersolution defined throughout $\R^2$, and Lemma \ref{posMul0} implies that there exists a positive multiplier $\phi$ in $B_m$.
Moreover,
\begin{align*}
\norm{\tilde u}_{L^\iny\pr{B_m}}
&\le \sup\set{ \exp\pr{C_0 \abs{z_0 + d^{-1} R z}} : z \in B_m}
\le \exp\pr{C_0 \pr{1 + d^{-1}m} R} \\
\norm{\tilde u}_{L^\iny\pr{B_d}}
&\ge \abs{u\pr{0}}
\ge 1.
\end{align*}
Therefore, assuming that $r$ is sufficiently small, we may apply Theorem \ref{OofV} to $\tilde u$ with $K = \hat \mu d^{-1} R$, $F\pr{K} = \frac 1 5$, $C_1 = C_0 \pr{d+m}/\hat \mu$, $c_1 = 0$, and $p =0$ to get
\begin{align*}
\norm{u}_{L^\iny\pr{B_{rd^{-1}R}\pr{z_0}}} 
= \norm{\tilde u}_{L^\iny\pr{B_{r}}} 
\ge r^{5 C K}
= r^{5 C  \hat \mu d^{-1} R},
\end{align*}
where $C = C\pr{\la,\hat \mu,C_0}$.
Taking $r = d/R$ and further assuming that $R \ge \frac 1 d$ shows that
\begin{align*}
\norm{u}_{L^\iny\pr{B_{1}\pr{z_0}}} 
\ge \exp\pr{- 10 C  \hat \mu d^{-1} R \log R}.
\end{align*}
As $z_0$ was an arbitrary point for which $\abs{z_0} = R$, the conclusion follows.
\end{proof}

\section{The proof of Theorem \ref{LandisThm}}
\label{Landis}

The idea behind the proof of Theorem \ref{LandisThm} is very similar to that of Theorem \ref{LandisThm0}, but because the set of assumptions is different, the execution of the proof also differs.
As in the previous section, we begin with an iteration proposition.

\begin{prop}
\label{itProp}
Let the coefficient matrix $A$ be symmetric and satisfy \eqref{ellip}, \eqref{Abd}, and \eqref{Lips}.
Assume that $V : \R^2 \to \R$ satisfies \eqref{V+Cond} and \eqref{V-Cond} and $W : \R^2 \to \R^2$ satisfies \eqref{WCond}.
Let $u: \R^2 \to \R$ be a solution to \eqref{ePDE} for which \eqref{uBd} holds.
Suppose that for any $\ga \in\pr{0, \min\set{\eps_0, \eps_1, \eps_2}}$ and any $S \ge \tilde S\pr{\mu_0, \mu_1, \mu_2, \eps_0, \eps_1, \eps_2, C_0, c_0, \ga, \La}$, there exists an $\al \in \pb{1, 2}$ so that \eqref{stepnEst} holds.
Assume further that when we restrict to $z \in \R^2$ with $\abs{z} \in\brac{ \frac S 2,  4 S^{1 + \ga}}$, \eqref{ellip} and \eqref{Abd} hold with $\la \ge 1 - \frac {S^{-\ga}} {20}$.
Then for any $R \ge S + S^{1 + \ga} - \La$, \eqref{stepn1Est} holds where $C_2$ is a universal constant and $\be = \max\set{\frac \al {1 + \ga}, 1} + \frac \ga {1 + \ga}$.
\end{prop}

\begin{proof}
Fix $\ga \in \pr{0, \min\set{\eps_0, \eps_1, \eps_2}}$ and $S \ge \tilde S$.
Define $T = S^{1 + \ga}$.
Let $z_1 \in \R^2$ be such that $\abs{z_1} = S + T - \La= R$ and define $z_0 = S \frac{z_1}{\abs{z_1}}$.

With $a = \pr{1 - \frac{S}{5T}}^{-1}$, define $\tilde u\pr{z} = u\pr{z_1 + aT z}$, $\widetilde A\pr{z} = A\pr{z_1 + aT z}$, $\widetilde W\pr{z} = aT W\pr{z_1 + aT z}$, and $\widetilde V\pr{z} = \pr{aT}^2 V\pr{z_1 + aT z}$ so that
\begin{align*}
- \di \pr{\widetilde A\pr{z} \gr \tilde u\pr{z} }
&+ \widetilde W\pr{z} \cdot \gr \tilde u \pr{z} + \widetilde V\pr{z} \tilde u\pr{z} = 0 .
\end{align*}
Define $d = \si\pr{1 - \frac S {20 T}}$, $b = \rho\pr{1 + \frac S {20 T}}$, and $m  = b + \frac {S}{20 T}$. 
As in the proof of Proposition \ref{itProp0}, if we choose $S$ sufficiently large with respect to $\mu_0$ and $\ga - \eps_0$, then $d \ge 1 - \frac S{5T}$, $b \le 1 + \frac{S}{5T}$, and $a T m \le T + \frac{S}{2}- \La$.
Then \eqref{grtA} holds and conditions \eqref{V+Cond}, \eqref{V-Cond} and \eqref{WCond} imply that 
\begin{align*}
&\norm{\widetilde V_+}_{L^\iny\pr{B_m}} \le \pr{a T}^2 \\
&\norm{\widetilde V_-}_{L^\iny\pr{B_m}} \le \pr{\mu_1 aT}^2 \pr{S/2}^{-2\pr{1 + \eps_1}} \\
&\norm{\widetilde W}_{L^\iny\pr{B_m}} \le \mu_2 aT \pr{S/2}^{-\pr{1 + \eps_2}}.
\end{align*}
For Theorem \ref{OofV} to be applicable, we require a positive multiplier in $B_m$.
An application of Lemma \ref{posMul} will produce such a function, but we must have $\norm{\gr \widetilde a_{ij}}_{L^\iny\pr{B_m}} \le \frac {\la} {4m}$, $\norm{\widetilde V_-}_{L^\iny\pr{B_m}} \le \frac {\la} {m^2 + 1}$, and $\norm{\widetilde W}_{L^\iny\pr{B_m}} \le \frac {\la} {2m}$.
The bounds on $\la$ and $m$ imply there are universal constants $c_0$, $c_1$, and $c_2$ for which $c_0 \le \frac {\la} {4m}$, $c_1^2 \le \frac {\la} {m^2 + 1}$ and $c_2 \le  \frac {\la} {2m}$.
Combining these conditions with the bounds observed above, we require
\begin{align*}
T = S^{1 + \ga} &\le \min\set{
\frac{c_0 2^{1 - \eps_0} }{5 \mu_0} S^{1 + \eps_0},
\frac{c_1 2^{1 - \eps_1}}{5 \mu_1} S^{1 + \eps_1},
\frac{c_2 2^{1 - \eps_2}}{5 \mu_2 }S^{1 + \eps_2} }.
\end{align*}
Since $\ga < \min\set{\eps_0, \eps_1, \eps_2}$, if $S$ is sufficiently large with respect to $\mu_0$, $\mu_1$, $\mu_2$, $\eps_0$, $\eps_1$, $\eps_2$, $c_0$, $c_1$, $c_2$, and $\ga$, then this minimality condition will be satisfied, and we conclude that the required positive multiplier exists in $B_m$.

As in the proof of Proposition \ref{itProp0}, if $S$ is sufficiently large with respect to $C_0$ and $c_0$, then \eqref{tuUpper} holds and \eqref{tuLower} follows from assumption \eqref{stepnEst}.
In particular, all of the hypotheses of Theorem \ref{OofV} hold with $K = \frac{5} 4 T$, $F\pr{K} = \frac S {20 T} = \frac{1}{20 } \pr{\frac {4K}{5}}^{-\frac{\ga}{1 + \ga}}$, $C_1 = \frac 4 5$, $c_1 = \pr{\frac 4 {5}}^2$, and $p = \frac \al {1 + \ga} \le 2$.
That is, with $q = \max\set{1, p}$ and a universal constant $C$, we have
\begin{align*}
\norm{u}_{B_{a T r}\pr{z_1}} 
\ge \norm{\tilde u}_{B_r\pr{z_1}} 
\ge r^{C K^q/ F\pr{K}}
\ge r^{20 \tilde C T^{q + \frac{\ga}{1 + \ga}}},
\end{align*}
where $\tilde C = C \pr{\frac{5}4}^2$.
Setting $r = \frac {4\La}{5 T}$ shows that
\begin{align*}
\norm{u}_{B_{\La}\pr{z_1}} 
&\ge \exp\brac{-21 \tilde C T^{\be} \log T },
\end{align*}
where we have assumed that $\log T \ge 20 \log\pr{5/4\La}$.
Since $z_1 \in \R^2$ with $\abs{z_1} = R$ was arbitrary, the conclusion of the proposition follows.
\end{proof}

Now we repeatedly apply Proposition \ref{itProp} to prove Theorem \ref{LandisThm}.
Much of this proof resembles that of Theorem \ref{LandisThm0}, so we often refer to that proof.

\begin{proof}[The proof of Theorem \ref{LandisThm}]
Let $\eps > 0$ be given. 
Fix $R \ge R_0$, where $R_0$ will be specified below.

As in the proof of Theorem \ref{LandisThm0}, we may assume that $A$ is symmetric and a change of variables shows that there exists $\hat z \in \R^2$ with $\abs{\hat z} \ge R$ for which $\widetilde A\pr{\hat z} = I$.
Condition \eqref{Lips} implies that $\widetilde W - W_n$ in \eqref{ndForm} satisfies \eqref{WCond} with $\eps_2$ replaced by $\min\set{\eps_2, \eps_0}$ and $\mu_2$ replaced by $\mu_2 + \mu_0$.
An application of Theorem 1.1 from \cite{LW14} implies that there exists $\hat S_0 = \hat S_0\pr{C_0,\mu_0,\mu_1,\mu_2,\eps_0}$ so that whenever $S \ge \hat S_0$, it holds that
\begin{align*}
\inf_{\abs{z_0} = S}\norm{\tilde u}_{L^\iny\pr{B_{\sqrt \la}\pr{z_0}}} \ge \exp\pr{- C_3 S^{4/3} \pr{\log S}^{\eta \pr{S}}},
\end{align*}
where $\disp \eta\pr{S} = C' \pr{\log S} \pr{\log \log \log S} \pr{\log \log S}^{-2}$ and the constants $C_3$ and $C'$ both depend on $c_0$, $C_0$, $\mu_0$, $\mu_1$, $\mu_2$, $\eps_0$, and $\eps_2$.

Let $\ga := \min\set{\eps, \frac {\eps_0}{2}, \frac{\eps_1}{2}, \frac{\eps_2}{2}}$ and choose 
$$S_0 \ge \max\set{ \tilde S\pr{\mu_0, \mu_1, \mu_2, \eps_0, \eps_1, \eps_2, C_0, c_0, \ga, \sqrt \la}, \hat S_0, \pr{20 C_4}^{\frac 2{\eps_0}}},$$ 
where $\tilde S$ is as given in Proposition \ref{itProp} and $C_4$ is the specific constant from the proof of Theorem \ref{LandisThm0}.
If need be, choose $S_0$ even larger so that \eqref{logLarge} holds with $C_2$ from Proposition \ref{itProp} and the $C_3$ that we just introduced. 
Consequently, for every $S \ge S_0$, \eqref{step0} holds with $\al_0 = \frac 4 3 + \frac{\ga^2}{2}$.

As in the proof of Theorem \ref{LandisThm0}, condition \eqref{Lips}, the largeness of $S_0$, and  $\widetilde A\pr{\hat z} = I$ imply that $\la \ge 1 - \frac{S^{- \ga}}{20}$ when $\abs{z} \ge \frac S 2$.
Therefore, we are in a position to apply Proposition \ref{itProp}.

If $\al_0 \le 1 + \ga \le 1 + \eps$, then there is no need to iterate so we define $N = -1$.
Otherwise, we assume that $\al_0 > 1 + \ga$.
Define $S_{n+1} = S_n + S_n^{1 + \ga} - \sqrt \la$ and $\al_{n+1} = \frac{\al_n + \ga}{1 + \ga} + \frac{\ga^2}{2}$ for $n = 0, 1, \ldots, N$, where $N \in \N$ is defined so that $\al_N > 1 + \ga$, while $\al_{N+1} \le 1 + \ga$.
As in the proof of Theorem \ref{LandisThm0}, \eqref{step0} in combination with repeated applications of Proposition \ref{itProp} shows that for all $n = 1, \ldots, N+1$,
\begin{align*}
\inf_{\abs{z_n} = S_n}\norm{\tilde u}_{L^\iny\pr{B_{\sqrt \la}\pr{z_n}}} 
&\ge \exp\pr{- S_n^{\al_n}}.
\end{align*}
As argued previously, the iteration stops after a finite number of steps.
Reversing the change of variables, we see that for any $S \ge S_{N+1}$,
\begin{align*}
\inf_{\abs{z_0} = \la^{-1/2}S}\norm{u}_{L^\iny\pr{B_{1}\pr{z_0}}} 
&\ge \inf_{\abs{Q z_0} = S} \norm{\tilde u}_{L^\iny\pr{B_{\sqrt \la}\pr{Q z_0}}} 
\ge \exp\pr{- S^{\al_{N+1}}}
\ge \exp\pr{- S^{1 + \eps}},
\end{align*}
where we have used that $\ga \le \eps$.
If we define $R_0 := \la^{-1/2} S_{N+1}$, then the conclusion of the theorem follows.
\end{proof}

\def\cprime{$'$}


\end{document}